\documentclass{amsart}
\makeatletter
\def\@tocline#1#2#3#4#5#6#7{\relax
  \ifnum #1>\c@tocdepth 
  \else
    \par \addpenalty\@secpenalty\addvspace{#2}%
    \begingroup \hyphenpenalty\@M
    \@ifempty{#4}{%
      \@tempdima\csname r@tocindent\number#1\endcsname\relax
    }{%
      \@tempdima#4\relax
    }%
    \parindent\z@ \leftskip#3\relax \advance\leftskip\@tempdima\relax
    \rightskip\@pnumwidth plus4em \parfillskip-\@pnumwidth
    #5\leavevmode\hskip-\@tempdima
      \ifcase #1
       \or\or \hskip 1em \or \hskip 2em \else \hskip 3em \fi%
      #6\nobreak\relax
    \hfill\hbox to\@pnumwidth{\@tocpagenum{#7}}\par
    \nobreak
    \endgroup
  \fi}
\makeatother

\usepackage{amssymb}
\usepackage{mathrsfs}
\usepackage[colorlinks=true,linkcolor=blue,citecolor=magenta]{hyperref}

\newtheorem{thm}{Theorem}[section]

\newtheorem{lem}[thm]{Lemma}
\newtheorem*{lem*}{Lemma}

\newtheorem{prop}[thm]{Proposition}

\newtheorem{cor}[thm]{Corollary}

\newtheorem{question}[thm]{Question}

\theoremstyle{remark}

\theoremstyle{definition}

\newcounter{my_enumerate_counter}

\newcommand\comment[1]{}

\newcommand\Tcal{\mathcal{T}}

\newcommand\Nbb{\mathbb{N}}

\newcommand\Zbb{\mathbb{Z}}

\newcommand{\seq}[1]{\mathtt{#1}}

\renewcommand{\epsilon}{\varepsilon}

\title[von Neumann-Day problem]
{A geometric solution to the von {N}eumann-Day 
problem for finitely presented groups}



\title[A finitely presented infinite simple group]
{A finitely presented infinite simple group of homeomorphisms of the circle}


\thanks{The author thanks Coop City in Lausanne since the key idea behind the group was conjured while sipping a coffee in its beautiful terrace.
The author also thanks Masato Mimura, Nicol\'{a}s Matte Bon, Matt Brin and Michele Triestino for helpful discussions and comments, and the TIFR in Mumbai for a visit during which some of the ideas were developed.
This research has been supported by a Swiss national science foundation ``Ambizione" grant.}

\author{Yash Lodha}

\address{EPFL
\\ Lausanne\\ Switzerland}

\email{{\tt yash.lodha@epfl.ch}}

\begin{document}
\maketitle

\begin{abstract}
We construct a finitely presented, infinite, simple group that acts by homeomorphisms on the circle,
but does not admit a non-trivial action by $C^1$-diffeomorphisms on the circle.
The group emerges as a group of piecewise projective homeomorphisms of $\mathbf{S}^1=\mathbf{R}\cup \{\infty\}$.
However, we show that it does not admit a non-trivial action by piecewise linear homeomorphisms of the circle.
Another interesting and new feature of this example is that it produces a non amenable orbit equivalence relation with respect to the Lebesgue measure.
\end{abstract}

\tableofcontents

\section{Introduction}

Subgroups of the group of orientation preserving homeomorphisms of the circle provide important examples of finitely presented, infinite, simple groups.
The first such example was constructed by Thompson, which is the group known as Thompson's group $T$.
This example was generalised by Higman \cite{Higman}, and later by Brown-Stein in \cite{BrownStein}.
These groups emerge as groups of piecewise linear homeomorphisms of the circle.

Two fundamental facts about these groups and their actions are the following.
In \cite{GhysSergiescu}, Ghys and Sergiescu proved that $T$ admits a faithful action by $C^{\infty}$-diffeomorphisms of the circle. 
In fact, they show that the standard actions of $T$ are topologically conjugate to such an action.
Secondly, the standard actions (indeed all known actions) of the groups of Thompson, Higman, Brown-Stein on the circle also produce amenable equivalence relations,
despite being non amenable.

In this article we provide a construction of a group $S$, which illustrate the existence of the following phenomenon in contrast to previous examples. 

\begin{thm}
The group $S$ satisfies the following:
\begin{enumerate}
\item $S$ is a finitely presented, infinite, simple group.
\item $S$ admits a faithful action by homeomorphisms of the circle, however
\begin{enumerate}
\item[(i)] $S$ does not admit a non-trivial action by $C^1$-diffeomorphisms of the circle.
\item[(ii)] $S$ does not admit a non-trivial action by piecewise linear homeomorphisms of the circle.
\end{enumerate}
\item The prescribed action of $S$ on the circle produces a non amenable equivalence relation with respect to the Lebesgue measure.
\end{enumerate}
\end{thm}

To our knowledge, $S$ is the first finitely presented simple group to satisfy any of the properties $(2.\textup{i}), (2.\textup{ii})$ or $(3)$ above.
We remark that it was shown in \cite{BonattiLodhaTriestino} that certain Brown-Stein groups have no faithful $C^2$-action on the circle (See Theorem $3.3$ and Corollary $3.4$, where they are described as \emph{Thompson-Stein groups}).
However, it is not known whether they admit faithful $C^1$ actions.
Furthermore, we remark that $S$ contains free subgroups, and hence it is nonamenable.
However, producing a non amenable equivalence relation for the given action is a stronger property.

Recall that Thompson's group $T$ is the group of piecewise $PSL_2(\mathbf{Z})$ homeomorphisms of $\mathbf{S}^1=\mathbf{R}\cup \{\infty\}$ with breakpoints in the set $\mathbf{Q}\cup \{\infty\}$.
Our group $S$ is generated by $T$ together with the following piecewise projective homeomorphism of the circle:
\[{\bf s}(t)=
\begin{cases}
t&\text{ if }t\leq 0\\
 \frac{2t}{1+t}&\text{ if }0\leq t\leq 1\\
 \frac{2}{3-t}&\text{ if }1\leq t\leq 2\\ 
t&\text{ if }t\geq 2\\
\end{cases}
\]

In this article we are concerned with certain subgroups of $\textup{Homeo}^+(\mathbf{S}^1)$.
For examples of finitely presented, infinite, simple groups that \emph{do not} admit non-trivial actions on the circle,
we refer the reader to the works of Burger-Moses \cite{BurgerMoses}, Scott \cite{Scott}, Rover \cite{Rover}, Caprace-Remy\cite{CapraceRemy} and Nekresheyvich \cite{N}.
Finally, we remark that the group $S$ is not left orderable (since it contains torsion elements), hence we shall only be concerned with circle actions of $S$.

\section{Preliminaries}

All actions in this article will be right actions, unless otherwise specified or when function notation is used.
In this article we shall go back and forth between two descriptions of the group $S$.
The first description is as a group of piecewise projective homeomorphisms of the circle.
The second is as a group of homeomorphisms of the cantor set $2^{\mathbf{N}}$ of infinite binary sequences endowed with the product topology.
The two models of the group shall be semi-conjugate via a map $\Phi:2^{\mathbf{N}}\to \mathbf{R}\cup\{\infty\}$,
which is described below in \ref{ContinuedFractions}.

\subsection{Piecewise projective homeomorphisms}\label{pph}

In this section we describe various piecewise projective homeomorphisms of the real line and the circle,
which shall provide generating sets for the group $S$ as well as the groups $G_0,G,T,F,BB(1,2)$ which will play a role in the construction of $S$.
The notation here will be fixed throughout the article.
In particular, the letters $a,b,c,{\bf s}, l$ used to denote the maps below shall be fixed.

The group $F$ is the group of piecewise $PSL_2(\mathbf{Z})$ homeomorphisms of the real line
with breakpoints in the set $\mathbf{Q}$.
This is generated by the following:
$$
a(t)=t+1\qquad b(t)=\left\{\begin{array}{ll}
\smallskip
t&\text{ if }t\leq 0\\
\medskip
\dfrac{t}{1-t}&\text{ if }0\leq t\leq \dfrac12\\
\medskip
\dfrac{3t-1}{t}&\text{ if }\dfrac12\leq t\leq 1\\
t+1&\text{ if }1\leq t\end{array}\right.
$$

The group $T$ is the group of piecewise $PSL_2(\mathbf{Z})$ homeomorphisms of $\mathbf{S}^1=\mathbf{R}\cup \{\infty\}$
with breakpoints in the set $\mathbf{Q}\cup \{\infty\}$.
This is generated by $F$ together with the involution: $$l(t)=-\frac{1}{t}$$

The group $G_0$ defined in \cite{LodhaMoore} is generated by $F$ together with the following piecewise projective homeomorphism:
\[c(t)=
\begin{cases}
t&\text{ if }t\leq 0\\
 \frac{2t}{1+t}&\text{ if }0\leq t\leq 1\\
t&\text{ if }t\geq 1\\
\end{cases}
\]

Next, we define the following maps:

\[d_1(t)=
\begin{cases}
t&\text{ if }t\leq 0\\
 2t&\text{ if }t\geq 0\\
\end{cases}
\qquad
d_2(t)=
\begin{cases}
2t&\text{ if }t\leq 0\\
t&\text{ if }t\geq 0\\
\end{cases}
\]

The group $G$ defined in \cite{LodhaMoore} is generated by $F$ together with $c, d_1,d_2$.
Finally, recall from the introduction that the group $S$ is generated by $T$ together with the map ${\bf s}$ from the introduction.

\subsection{Binary sequences}

To describe the actions on the cantor set,
we need to fix some notation.
We will take $\mathbf{N}$ to include $0$. 
Let $2^\Nbb$ denote the collection of all infinite binary sequences and
let $2^{<\Nbb}$ denote the collection of all finite binary sequences.
For $s\in 2^{<\mathbf{N}}$, $s(i)$ denotes the $i$'th digit of $s$.
If $i \in \Nbb$ and $u$ is a binary sequence of length at least $i$, we will let 
$u\restriction i$ denote the initial part of $u$ of length $i$.
We denote by $|s|$ the length of $s$,
which is the number of digits in $s$.

If $s$ and $t$ are finite binary sequences, then we will write $s \subseteq t$
if $s$ is an initial segment of $t$ and $s \subset t$ if $s$ is a proper
initial segment of $t$.
If neither $s \subset t$ nor $t \subset s$, then we will say that
$s$ and $t$ are \emph{independent}.
Note that in particular if $s=t$ then $s,t$ are independent.
A list of finite binary sequences $s_1,...,s_n$ is said to be \emph{independent},
if they are pairwise independent.

The set $2^{<\Nbb}$ is equipped with an order defined
by $s < t$ if $t \subset s$ or $s$ and $t$ are independent
and $s(i) < t(i)$ where $i$ is the smallest number such that $s(i) \ne t(i)$.
If $u, s_1,...,s_n\in 2^{<\mathbf{N}}$, we say that
$u$ \emph{dominates} $s_1,...,s_n$ if for each $1\leq i\leq n$,
either $s_i,u$ are independent or $s_i\subset u$.

The finite binary sequences $s,t$ are said to be \emph{consecutive}
if there is a binary sequence $u$ and numbers $n_1,n_2\in \mathbf{N}$
such that $s=u 0 1^{n_1}$ and $t=u 1 0^{n_2}$. 
(Note that $1^0, 0^0$ is assumed to represent the empty sequence).
In other words, $s,t$ are consecutive if there is a finite rooted binary tree 
such that $s,t$ occur as consecutive leaves of the tree, in the order defined above.

A list of finite binary sequences $s_1,...,s_n$ is said to be \emph{consecutive}
if each pair $s_i,s_{i+1}$ is consecutive for $1\leq i\leq n-1$.
Note that if $s_1,...,s_n$ are consecutive then they are automatically independent. 

The sequences $s,t$ are said to be \emph{cyclically consecutive}
if either $s,t$ are consecutive or $s=0^{k_1},t=1^{k_2}$ for some $k_1,k_2\in \mathbf{N}\setminus \{0\}$.
We remark that if $s,t$ are cyclically consecutive then there is a finite rooted binary tree $T$ such that
$s,t$ are leaves of $T$ that are either consecutive or it holds that $s$ is the leftmost leaf and $t$ is the rightmost leaf  of $T$.

If $\xi$ and $\eta$ are infinite binary sequences, then we will say that
$\xi$ and $\eta$ are \emph{tail equivalent} if there are $s$, $t$, and $\zeta$ such that
$\xi = s \zeta$ and $\eta = t \zeta$.
We use $0^{\infty}, 1^{\infty}$ to denote the constant infinite sequences
$000....$ and $111...$ respectively.
More generally, given a finite binary sequence $s$, $s^{\infty}$ denotes
the sequence $sss....$.
A sequence $\psi\in 2^{\mathbf{N}}$ is called \emph{rational}
if it is tail equivalent to $1^{\infty}$ or $0^{\infty}$.
Otherwise, it is called \emph{irrational}.

We denote by $\Tcal$ the collection of all finite rooted binary trees.
A tree $\Gamma$ in $\Tcal$ will be denoted by a set of finite binary sequences $s_1,...,s_n$
which are the addresses of leaves in $\Gamma$.
The indices are ordered so that if $i<j$ then $s_i<s_j$.
 We view elements $\Gamma$ of $\Tcal$ as \emph{prefix} sets.
So we view $\Gamma$ as a set of finite binary sequences with the property that every infinite 
binary sequence has a unique initial segment in $\Gamma$. 

A \emph{tree diagram} is a pair $(L,R)$ of elements of $\Tcal$
with the property that $|L|=|R|$.
A tree diagram describes a map of infinite binary sequences
as follows:
$$
s_i \xi \mapsto t_i \xi
$$
where $s_i$ and $t_i$ are the $i$th elements of $L$ and $R$, respectively, in the
order as defined above and $\xi$ is any binary sequence.
The collection of all such functions from $2^\Nbb$ to $2^\Nbb$ defined
in this way, under the operation of composition, is \emph{Thompson's group $F$}.
The function associated to a tree diagram
is also defined on any finite binary sequence $u$ such that
$u$ has a prefix in $L$.
So the group $F$ admits a partial action on $2^{<\mathbf{N}}$.
Given $f\in F$ and $s\in 2^{<\mathbf{N}}$, we say that $f$ \emph{acts on} $s$ if $f(s)$ is defined.
Similarly, given $s_1,...,s_n\in 2^{<\mathbf{N}}$ we say that $f$ \emph{acts on} $s_1,...,s_n$ if $f(s_1),...,f(s_n)$
are all defined.

\subsection{Generators for $T$}

For each $n\in \mathbf{N}$ and $\xi\in 2^{\mathbf{N}}$, we define the maps
$$
\xi\cdot p_0=
\begin{cases}
1(\eta) & \textrm{ if } \xi = 0 (\eta) \\
0 (\eta) & \textrm{ if }\xi=1(\eta)\\
\end{cases}
$$
$$\text{ If }n>0\qquad \xi\cdot p_n=
\begin{cases}
1^{k+1}0(\eta) & \textrm{ if } \xi = 1^{k}0 \eta\text{ for }0\leq k\leq n-1 \\
1^{n+1} (\eta) & \textrm{ if }\xi=1^{n}0 (\eta)\\
0 (\eta) & \textrm{ if }\xi=1^{n+1}(\eta)\\
\end{cases}
$$

The group $T$ is then generated by $F$ together with the generators $\{p_0,p_1,...\}$.
$F,T$ admit the above action on infinite binary sequences, but they also admit a partial
action on the set of nonempty finite binary sequences.
This action is given in exactly the same fashion, via prefix replacement maps,
whenever defined.
For example, $p_0$ acts on all nonempty finite binary sequences,
and $x_1$ acts on all nonempty finite sequences except the sequence $0$.

The following is an elementary and well known observation about this action.

\begin{lem}\label{Taction}
The partial action of $T$ on the set of nonempty finite binary sequences is transitive.
Moreover the following holds:
\begin{enumerate}
\item Given any two cyclically consecutive pairs $\sigma_1,\sigma_2$ and $\tau_1,\tau_2$ of finite binary sequences,
there is an element $f\in T$ such that $$f(\sigma_1)=\tau_1\qquad f(\sigma_2)=\tau_2$$
\item Consider independent, non-cyclically consecutive pairs $\sigma_1,\sigma_2$ and $\tau_1,\tau_2$ of finite binary sequences,
such that $\sigma_1\neq \sigma_2$ and $\tau_1\neq \tau_2$. 
Then there is an element $f\in T$ such that $$f(\sigma_1)=\tau_1\qquad f(\sigma_2)=\tau_2$$
\end{enumerate}
\end{lem}

More generally, the following holds:

\begin{lem}\label{Taction2}
Denote by $\mathcal{L}_n$ as the set of $n$-tuples of independent finite binary sequences.
The the partial action of $T$ on $\mathcal{L}_n$ has finitely many orbits.
\end{lem}

\subsection{The continued fractions model}\label{ContinuedFractions}

Now we relate the description of the groups in the piecewise projective setup with the 
actions on the cantor set.
This will be done by means of semiconjugation.
Consider the following maps:

$$
\begin{array}c
\medskip\varphi:2^\mathbf{N}\longrightarrow[0,\infty]\\
\begin{array}l
\medskip\varphi(\mathtt0\xi)=\dfrac{1}{1+\frac{1}{\varphi(\xi)}}\\
\varphi(\mathtt1\xi)=1+\varphi(\xi)
\end{array}
\end{array}
\qquad
\begin{array}c
\Phi:2^\mathbf{N}\longrightarrow \mathbf{R}\cup \{\infty\}\\
\begin{array}l
\Phi(\mathtt0\xi)=-\varphi(\tilde\xi)\\
\Phi(\mathtt1\xi)=\varphi(\xi)
\end{array}
\end{array}
$$
where $\tilde\xi$ is the sequence obtained from $\xi$ by replacing all symbols $\mathtt0$ by $\mathtt1$ and viceversa. 

Consider the following two primitive functions:
\[
(\xi)\cdot x
 =
\begin{cases}
\seq{0}\eta & \textrm{ if } \xi = \seq{00} \eta \\
\seq{10}\eta & \textrm{ if } \xi = \seq{01} \eta \\
\seq{11}\eta & \textrm{ if } \xi = \seq{1} \eta \\
\end{cases}
\qquad
(\xi)\cdot y =
\begin{cases}
\seq{0}y(\eta) & \textrm{ if } \xi = \seq{00} \eta \\
\seq{10}y^{-1}(\eta) & \textrm{ if } \xi = \seq{01} \eta \\
\seq{11}y(\eta) & \textrm{ if } \xi = \seq{1} \eta \\
\end{cases}
\]

From these functions, we define families of functions $x_s$ and 
$y_s$ for $s \in 2^{<\Nbb}$.
These act just as $x$ and $y$, but localised to those binary sequences which extend $s$.
\[
(\xi)\cdot x_s
 =
\begin{cases}
s (\eta\cdot x) & \textrm{ if } \xi = s \eta \\
\xi & \textrm{otherwise}
\end{cases}
\qquad
(\xi)\cdot y_s
 =
\begin{cases}
s (\eta\cdot y) & \textrm{ if } \xi = s \eta \\
\xi & \textrm{otherwise}
\end{cases}
\]

We read $x_{1^0}$ as $x$ in the above.
The elements $$\{y_s\mid s\in 2^{<\mathbf{N}}\}$$
shall be denoted as \emph{percolating elements}.
Next, we have the following result (Proposition 3.1 in \cite{LodhaMoore}):

\begin{prop} 
For all $\xi$ in $2^\mathbf{N}$ we have
$$
\Phi(\xi)\cdot a=\Phi(\xi\cdot x)\qquad \Phi(\xi)\cdot b=\Phi(\xi \cdot x_{\mathtt1})\qquad \Phi(\xi)\cdot c=\Phi(\xi\cdot y_{\mathtt{10}})
$$
\end{prop}

The map $$t\to \frac{2}{3-t}\restriction [2,3]$$
is a conjugate of $$\frac{2t}{1-t}\restriction [0,1]$$ by $t\to t+1$.
Since $$\frac{2t}{1-t}\restriction [0,1]=c^{-1}\restriction [0,1]$$
we obtain the following.

\begin{prop}
For all $\xi\in 2^{\mathbf{N}}$ we have:
$$\Phi(\xi)\cdot {\bf s}=\Phi(\xi\cdot (y_{10}y_{110}^{-1}))$$
\end{prop}

Finally, it is an easy exercise to show the following: 

\begin{prop}
$\Phi(\xi)\cdot l=\Phi(\xi\cdot p_0)$
\end{prop}

It follows that the group $S$ is generated by $x,x_1,p_0, y_{10}y_{110}^{-1}$.

\subsection{Standard forms for $G_0,G$}

We recall from \cite{LodhaMoore} a standard form for the elements of $G_0$ and $G$,
which is a word that represents a given group element with desirable structure.
An element $g$ is said to be a \emph{standard form} if it is of the form $$fy_{s_1}^{t_1}...y_{s_n}^{t_n}$$
such that
\begin{enumerate}
\item $f\in F$ and $t_i\in \mathbf{Z}\setminus \{0\}$.
\item $s_1,...,s_n$ satisfy that if $i<j$ 
then either $s_i,s_j$ are independent or $s_j\subset s_i$.
\end{enumerate}

Note that in \cite{Lodha}, we use a slightly stronger notion of a standard form,
and require that $s_i<s_j$ whenever $i<j$ in the above.
This is needed for providing normal forms for the groups $G_0,G$.
However in this article we shall not be concerned with normal forms and hence the weaker notion 
of standard forms will suffice.

In \cite{LodhaMoore} we showed the following:

\begin{lem}
Any element in the groups $G_0$ and $G$ can be represented as a word in standard form.
\end{lem}

We remark that imposing additional conditions on the above standard form provides a \emph{normal form},
which is a unique word representing a given group element, with desirable properties.
This was shown in \cite{Lodha}, however we shall not require the use of this in the present article.

Associated with a standard form $fy_{s_1}^{t_1}...y_{s_n}^{t_n}$ and a sequence $\sigma\in 2^{\omega}$
is the notion of a \emph{calculation}.
This is an infinite string in letters $y,y^{-1},\seq{0},\seq{1}$.
The evaluation of $fy_{s_1}^{t_1}...y_{s_n}^{t_n}$ on $\sigma$ comprises of a 
prefix replacement that is determined by the transformation $f\in F$
followed by an infinite sequence of
applications of the transformations described by the percolating elements.
The latter is encoded as an infinite string with letters $y,y^{-1},0,1$ and denoted as the calculation of
$y_{s_1}^{t_1}...y_{s_n}^{t_n}$ on $f(\sigma)$.
The calculation is equipped with the following substitutions:
\[
y\seq{00} \rightarrow \seq{0}y 
\qquad
y\seq{01} \rightarrow \seq{10}y^{-1}
\qquad
y\seq{1} \rightarrow \seq{11}y
\]
\[
y^{-1} \seq{0} \rightarrow \seq{00}y^{-1}
\qquad
y^{-1} \seq{10} \rightarrow \seq{01}y
\qquad
y^{-1} \seq{11} \rightarrow \seq{1}y^{-1}
\]
For example, for the word $y_{100}^{-1}y_{10}$
and the binary sequence $1001111...$
the calculation string is $10y0y^{-1}1111...$.
The output of the evaluation of the word on the binary string
is the limit of the strings obtained from performing these substitutions.
In this way, we define the calculation of a standard form (up to commutation) $y_{s_1}^{t_1}...y_{s_n}^{t_n}$ on an infinite binary sequence $\sigma$.
The calculation of a standard form (up to commutation) $fy_{s_1}^{t_1}...y_{s_n}^{t_n}$
on an infinite binary sequence $\sigma$ is defined to be the calculation
of $y_{s_1}^{t_1}...y_{s_n}^{t_n}$ on $f(\sigma)$.

The set of all such strings will be denoted as $\{0,1,y,y^{-1}\}^{\mathbf{N}}$. 
A calculation is an element of this set with the property that there are only finitely many occurrences of $y^{\pm}$.

A calculation has a \emph{potential cancellation} if upon performing a finite set of
substitutions we encounter a substring of the form $yy^{-1}$ or $y^{-1}y$.
When a calculation has no potential cancellations, we say that it has
\emph{exponent} $n$ if $n$ is the number of occurrences of the symbols $y^{\pm}$.
Note that there is no potential cancellation in the example above, and the exponent
is $2$.
A standard form in $G$ is said to contain a potential cancellation if there is an infinite binary sequence $\psi$
such that the associated calculation of the standard form on $\psi$ contains a potential cancellation.
The following was proved in \cite{LodhaMoore}. (Lemma $5.9$.)
We shall use the following version (Lemma $3.19$ in \cite{Lodha}).

\begin{lem}\label{taileq}
Let $fy_{s_1}^{t_1}...y_{s_n}^{t_n}$ be a standard form that does not contain potential cancellations.
Let $U\subset 2^{\mathbb{N}}$ be the support of $y_{s_1}^{t_1}...y_{s_n}^{t_n}$.
\begin{enumerate}
\item $fy_{s_1}^{t_1}...y_{s_n}^{t_n}$ does not preserve tail equivalence on a dense subset of $U\cdot f^{-1}$.
More particularly, given an open set $U'$ in $U\cdot f^{-1}$, there is an infinite binary sequence $\tau\in U'$ such that 
$\tau\cdot fy_{s_1}^{t_1}...y_{s_n}^{t_n}$ and $\tau$ are not tail equivalent.
\item $fy_{s_1}^{t_1}...y_{s_n}^{t_n}$ preserves tail equivalence on $2^{\mathbb{N}}\setminus f^{-1}(U)$.
\end{enumerate}
\end{lem}

\subsection{Normal subgroup structure}

We shall use the following results proved in \cite{BurilloLodhaReeves}.

\begin{thm}\label{simplegroup}
$G_0'$ is simple and every proper quotient of $G_0$ is abelian.
\end{thm}

We shall also use the following characterisation of elements of $G_0'$, viewed as standard forms in $G_0$.

\begin{lem}\label{charelementsG0prime}
Any element of $G_0'$ can be represented as a standard form $fy_{s_1}^{t_1}...y_{s_n}^{t_n}$
such that $\sum_{1\leq i\leq n} t_i=0$ and $s_1,...,s_n$ are non constant finite binary sequences.
Conversely, any such standard form represents an element in $G_0'$.
\end{lem}

\subsection{The obstruction to $C^1$-smoothability}

For $n\in \mathbb{N}, n>1$, we define the group $BB(1,n)$ as the group generated by the map $\eta(t)=t+1$
together with the maps:

\[\nu_1(t)=
\begin{cases}
t&\text{ if }t\leq 0\\
n t&\text{ if }t\geq 0\\
\end{cases}
\qquad
\nu_2(t)=
\begin{cases}
n t&\text{ if }t\leq 0\\
t&\text{ if }t\geq 0\\
\end{cases}
\]

This group is denoted as $G_{n}$ in \cite{BonattiLodhaTriestino},
but in the present article we denote it as $BB(1,n)$,
particularly to avoid confusion with the groups $G_0,G$.
In \cite{BonattiLodhaTriestino} (Theorem $6.9$), we showed the following:

\begin{thm}\label{BLTmain}
Let $n\in \mathbf{N},n>1$. 
Then the group $BB(1,n)$ does not admit a faithful $C^1$-action on the closed interval $[0,1]$. 
\end{thm}

The following is remarked in \cite{BonattiLodhaTriestino}.

\begin{thm}\label{mainD}
The group $BB(1,2)$ does not admit a faithful $C^1$-action on the circle.
\end{thm}
In \cite{BonattiLodhaTriestino} we do not provide the full details for the above.
For the sake of completeness, we include a full proof of this below.
We shall use the following results (which are similarly stated as Theorem $6.2$ and Proposition $6.1$ in \cite{BonattiLodhaTriestino}).

\begin{thm}\label{maindyn}
Let $n\in \mathbf{N},n>1$. 
Consider a faithful $C^1$-action of $BS(1,n)$ on $[0,1]$, without global fixed points in $(0,1)$. Then the action is
topologically conjugate to the standard affine action.
Moreover, assume that this group is generated by two generators $a,b\in \textup{Diff}^1([0,1])$ such that $a^{-1} b a= b^{n}$.
Then $a$ has a derivative equal to $n$ at its interior fixed point.
\end{thm}

We shall also use the following special case of a result due to Guelman and Louisse (see \cite{GuelmanLiousse}).

\begin{thm}\label{power}
Fix $n\in \mathbf{N},n>1$.
Let $BS(1,n)$ be generated by two generators $a,b\in \textup{Diff}^1(\mathbf{S}^1)$ such that $a^{-1} b a= b^{n}$.
Then there is an $m\in \mathbb{N}\setminus \{0\}$ such that the group generated by $a^m, b^{n-1}$
admits a global fixed point on $\mathbf{S}^1$.
\end{thm}

Now we proceed to prove Theorem \ref{mainD}.

\begin{proof}
Let $BB(1,2)$ be faithfully generated by three generators $\eta,\nu_1,\nu_2\in \textup{Diff}^1(\mathbf{S}^1)$,
which correspond (algebraically) to the generators $\eta,\nu_1,\nu_2$ in the description of $BB(1,2)$ defined in the beginning of this section.
Then $BS(1,2)$ is the subgroup generated by $\eta, \nu_1\nu_2$.
We fix the notation $\nu=\nu_1\nu_2$.

Using Theorem \ref{power} we conclude that there is an $m\in \mathbf{N}$ such that the subgroup $\langle \nu^m, \eta\rangle\cong BS(1,2^m)$ admits a global fixed point on the circle,
which for convenience we denote by $\infty$.
Consider the subgroup $BB(1,2^m)$ that is generated by $\eta,\nu_1^m,\nu_2^m$.
Since we have a faithful action of $BS(1,2^m)$ on $\mathbf{S}^1$ with a global fixed point,
there is a closed interval $I$ such that the restriction of the action of $BS(1,2^m)$ to $I$ is faithful and without a global fixed interior point.
By an application of \ref{maindyn} we obtain that there is a point $x$ in the interior of $I$ that is fixed by $\nu^m$ and so that $(\nu^m)^{'}(x)=2^m$.

Now we know that $[\nu_i^m,\nu^m]=1$ in $BB(1,2^m)$ for each $i\in \{1,2\}$.
This means that $$x\cdot ((\nu_i^m)^{-1} \nu^m \nu_i^m)$$ is a fixed point for $\nu^m$ so that $(\nu^m)^{'}(y)=2^m$ for $y=\nu_i^m(x)$.
The set $$\{t\in \mathbf{S}^1\mid \nu^m(t)=t, (\nu^m)^{'}(t)=2^m\}$$
must be finite, since otherwise this set admits a limit point where the derivative of $\nu^m$ is simultaneously equal to $1$ (since it is an accumulation point of fixed points of $\nu^m$)
and $2^m$ (since $\nu^m$ is $C^1$).

It follows that $\nu_i^m$ has a finite orbit containing $x$ for each $i\in \{1,2\}$.
So there is a number $n\in \mathbf{N}$ such that $\nu_1^{nm},\nu_2^{nm}$ fix $x$.
Now we generate a group with $$\eta, \nu_1^{mn},\nu_2^{mn},\nu^{mn}$$
to obtain a faithful $C^1$ action of $BB(1,2^{mn})$ for which:
\begin{enumerate}
\item The subgroup $BS(1,2^{mn})$ admits a global fixed point at $\infty$.
\item The generator $\nu_i^{mn}$ fixes a point on the circle for each $i\in \{1,2\}$.
\end{enumerate}
{\bf Claim}: Either one of the following holds:
\begin{enumerate}
\item $\nu_1^{mn},\nu_2^{mn}$ both fix $\infty$. 
\item $\nu_1^{mn}\mid I=\nu_2^{-mn}\mid I$ for some open interval $I\subset \mathbf{S}^1$ that contains $\infty$.
\end{enumerate}

Proof of claim: If $(1)$ were not the case, then at least one of the sequences $$(\infty\cdot (\nu_1^{mn})^l)_{l\in \mathbf{N}} \qquad ( \infty\cdot (\nu_2^{mn})^l)_{l\in \mathbf{N}}$$
would accumulate to a point other than $\infty$. 
Recall that $$\nu^l=\nu_1^l\nu_2^l\qquad \forall l\in \mathbb{N}$$ and $$[\nu_1,\nu]=[\nu_2,\nu]=[\nu_1,\nu_2]=1$$
Using this we note that the components of support of $\nu_1,\nu_2$ containing $\infty$ must coincide, and the restrictions of the maps on these components must be the inverses of each other.
This follows immediately by examining the components of supports of $\nu_1,\nu_2$ whose closure contains $\infty$.
Note that in the latter case of the claim, it is straightforward to check that the action $$BB(1,2^{mn})\mid I$$ is abelian, and hence $$BB(1,2^{mn})\mid (\mathbf{S}\setminus I)$$ is faithful.
In both cases, this provides a faithful $C^1$ of $BB(1,2^{mn})$ on a closed interval (either $\mathbf{S^1}\setminus I$) or the two point compactification of $\mathbf{S}^1\setminus \{\infty\}$), which cannot exist thanks to \ref{BLTmain}.
So we obtain a contradiction and hence our claim that $BB(1,2)$ admits a faithful $C^1$ action on the circle must be false.
\end{proof}

\section{A combinatorial model for the group $S$}

\subsection{An infinite presentation}\label{relations}

For any pair $\sigma,\tau$ of independent finite binary sequences, 
we define $$w_{\sigma,\tau}=y_{\sigma}y_{\tau}^{-1}$$
Note that $w_{\sigma,\tau}=w_{\tau,\sigma}^{-1}$.
In this notation we allow the situation $\sigma=\tau$,
for which $w_{\sigma,\tau}$ shall be a representative for the trivial element.
(Recall that by our convention, if $\sigma=\tau$, then $\sigma,\tau$ are assumed to be independent.)

\begin{lem}
For each pair $\sigma,\tau$ of independent finite binary sequences, 
$w_{\sigma,\tau}\in S$.
\end{lem} 

\begin{proof}
Recall from Lemma \ref{Taction} that the partial action
of $T$ on pairs of independent finite binary sequences has precisely three orbits.
A set of representatives of the orbits is given by $10,110$ and $10,1110$ and $10,10$ (for the case where $\sigma=\tau$).
We shall only treat the case where $\sigma\neq \tau$, since in the other case $w_{\sigma,\tau}$ is simply a representative of the trivial homeomorphism.

By Lemma \ref{Taction} there is an element $g\in T$ such that either $$g^{-1} w_{\sigma,\tau} g=w_{10,110}\qquad \text{ or }\qquad g^{-1} w_{\sigma,\tau} g=w_{10,1110}$$
Since $w_{10,110}$ is a generator of $S$, it remains to show that $w_{10,1110}\in S$.
Indeed it suffices to show that $w_{10,1101}\in S$, since there is a $g\in T$ such that $g^{-1} w_{10,1101} g=w_{10,1110}$.

Observe that $$w_{10,110}=x_{110}^{-1} w_{10,1101} w_{11001,11000}$$ 
where equality denotes equality as homeomorphisms of $2^{\mathbb{N}}$.
This follows from the equality $$y_{110}^{-1}=x_{110}^{-1} y_{1101}^{-1} y_{11001}y_{11000}^{-1}$$
Since we know that $x_{110}^{-1}, w_{11001,11000}\in S$, we conclude that $w_{10,1101}\in S$ as required.
\end{proof}

We shall use the following infinite generating set $\mathcal{X}$ for the group $S$:

\begin{enumerate}
\item ($w$-generators) $w_{\sigma,\tau}$ for $\sigma,\tau$ independent pairs of finite binary sequences.
\item ($x$-generators) $x_{\sigma}$ for $\sigma\in 2^{<\mathbf{N}}$ or $\sigma=\emptyset$.
\item ($p$-generators) $p_n$ for $n\in \mathbb{N}$.
\end{enumerate}

We now list a set of relations $\mathcal{R}$ in the generating set $\mathcal{X}$.
We separate them in three families as follows.\\

First, the relations in the $x$ generators:

\begin{enumerate}
\item $x_s^2=x_{s0}x_sx_{s1}$.
\item $x_{s}x_t=x_tx_{s\cdot x_t}$ if $x_t$ acts on $s$.
\end{enumerate}
Secondly, the relations in the $x$ and $p$ generators:
\begin{enumerate}
\item[(3)] $$x_{1^m}^{-1} p_n x_{1^{m+1}}=p_{n+1}\text{ if }n<m\qquad p_n x=p_{n+1}^2$$ 
$$p_n=x_{1^n}p_{n+1}\qquad p_n^{n+2}=1_S$$
\end{enumerate}

Finally,

\begin{enumerate}
\item[(4)] $w_{\sigma,\tau} x_{s}=x_s w_{\sigma \cdot x_s, \tau\cdot x_s}$ if $x_s$ acts on $\sigma,\tau$.
\item[(5)] $w_{\sigma,\tau} p_n=p_n w_{\sigma \cdot p_n, \tau\cdot p_n}$ if $p_n$ acts on $\sigma,\tau$.
\item[(6)] $w_{\sigma_1,\tau_1}w_{\sigma_2,\tau_2}=w_{\sigma_2,\tau_2}w_{\sigma_1,\tau_1}$ if $\sigma_1,\tau_1,\sigma_2,\tau_2$ are independent.
\item[7)] $w_{\sigma_1,\tau_1}w_{\sigma_2,\tau_2}=w_{\sigma_1,\tau_2}w_{\sigma_2,\tau_1}$ if $\sigma_1,\tau_1,\sigma_2,\tau_2$ are independent.
\item[(8)] $w_{\sigma,\tau}=x_{\sigma}w_{\sigma 0,\sigma 10}w_{\sigma 11,\tau}$ and $w_{\sigma,\tau}=x_{\tau}^{-1} w_{\sigma, \tau 1} w_{\tau 01,\tau 00}$.
\item[(9)] $w_{\sigma,\tau}w_{\tau,\nu}=w_{\sigma,\nu}$ if $\sigma,\nu$ are independent.
\item[(10)] $w_{\sigma,\sigma}=1_S$.
\end{enumerate}

The $x$-words together with relations $(1)-(2)$ provide an infinite presentation for $F$ (see \cite{LodhaMoore} for instance).
The $x,p$-words together with relations $(1)-(3)$ provide an infinite presentation for $T$.
(See \cite{BurilloClearyTaback} for instance.)
It shall be the main goal of this section to prove that the $x,p,w$-words together with relations $(1)-(10)$ provide an infinite presentation for $S$.
Subsequently, we shall prove that this reduces to a finite presentation for $S$.

Applications of the above relations to a given word shall be denoted as a \emph{move}.
For instance, consider a word $W$ in the generators containing a subword $V$. 
Then the move $V\to V'$ represents the replacement of the subword $V$ by the subword $V'$ in $W$,
corresponding to the relation $V=V'$.
We list the moves we shall require below.

\begin{enumerate}
\item (Rearrangement move) $$w_{\sigma,\tau} x_{s}\to x_s w_{\sigma \cdot x_s, \tau\cdot x_s}\qquad \text{ if } x_s \text{ acts on }\sigma,\tau$$
$$w_{\sigma,\tau} p_{T}\to p_T w_{\sigma \cdot p_T, \tau\cdot p_T} \qquad \text{ if } p_T\text{ acts on }\sigma,\tau$$
\item (Commuting move) $$w_{\sigma_1,\tau_1}w_{\sigma_2,\tau_2}\to w_{\sigma_2,\tau_2}w_{\sigma_1,\tau_1}\qquad \text{ if }\sigma_1,\tau_1,\sigma_2,\tau_2\text{ are independent}$$
\item (Relabelling move) $$w_{\sigma_1,\tau_1}w_{\sigma_2,\tau_2}\to w_{\sigma_1,\tau_2}w_{\sigma_2,\tau_1}\qquad \text{ if }\sigma_1,\tau_1,\sigma_2,\tau_2\text{ are independent}$$
\item (Amplification move) $$w_{\sigma,\tau}\to x_{\sigma}w_{\sigma 0,\sigma 10}w_{\sigma 11,\tau}$$
$$w_{\sigma,\tau}\to x_{\tau}^{-1}w_{\sigma,\tau 1}w_{\tau 01,\tau 00}$$
\item (Cancellation move) $$w_{\sigma,\tau}w_{\tau,\nu}\to w_{\sigma,\nu}\text{ if }\sigma,\nu\text{ are independent}$$
$$w_{\sigma,\sigma}\to \emptyset$$
\end{enumerate}

We shall fix some informal notation concerning the above moves which shall be useful in describing the arguments.
In the amplification moves (for instance, in the former case), the elements $w_{\sigma 0,\sigma 10},w_{\sigma 11,\tau}$ are said to be \emph{offsprings} of $w_{\sigma,\tau}$.
For the rearrangement moves, we shall say that $x_s$ (or $p_T$) \emph{acts on} the $w$-generator on its left.
This shall be understood to mean that it acts on the sequences in the subscript of the $w$-generator.
Moreover, upon moving this $x_s$ (or $p_T$) to the left of the $w$ generator, replacing $w_{\sigma,\tau}$ by $w_{\sigma\cdot x_s,\tau\cdot x_s}$ shall be described as \emph{reconfiguring} the $w$-generator.
We remark that the amplification move is analogous to the amplification move described in \cite{Lodha} for the groups $G_0,G$.
In general, the moves are natural analogs of the moves described in \cite{Lodha} and \cite{LodhaMoore}.

\subsection{Standard forms for $S$}\label{SSstdform}

Now we describe a notion of standard forms for the group $S$ in the infinite generating set above.

Recall that given a pair $s,t$ of finite binary sequences, we say that $s$ dominates $t$ if either $t\subseteq s$ or $s,t$ are independent.
Note that if $s,t$ are independent then $s$ dominates $t$ and $t$ dominates $s$.
Given pairs $s_1,t_1$ and $s_2,t_2$ of finite binary sequences, we say that $s_1,t_1$ \emph{dominates} $s_2,t_2$
if $u$ dominates $v$ for any $u\in \{s_1,t_1\}$ and $v\in \{s_2,t_2\}$.

We say that a word $f w_{s_1,t_1}^{l_1}...w_{s_n,t_n}^{l_n}$ is in standard form if:
\begin{enumerate}
\item $f\in T$ and $l_i>0$ for each $1\leq i\leq n$.
\item If $1\leq i<j\leq n$ then $s_i,t_i$ dominates $s_j,t_j$.
\end{enumerate}

The length of the standard form $fw_{s_1,t_1}^{l_1}...w_{s_n,t_n}^{l_n}$ is the quantity $\sum_{1\leq i\leq n}|l_i|$.
We define the quantity $$\textup{inf}\{|s_i|,|t_i|\mid 1\leq i\leq n\}$$ as the \emph{depth} of this standard form.
We fix the convention that the depth is $\infty$ if the standard form does not have any occurrences of a $w$-word. 
Two standard forms are said to be \emph{equivalent} if they represent the same elements of the group $S$.

\begin{lem}\label{stdform}
Any word in the infinite generating set $\mathcal{X}$ can be converted into a word in standard form using the relations in $\mathcal{R}$.
\end{lem}

Before we supply the proof of this Lemma, we need the following:

\begin{lem}\label{stddepth}
Given any standard form in $S$, and $m\in \mathbf{N}$, we can apply a sequence of moves defined in \ref{SSstdform}
to produce an equivalent standard form with depth at least $m$.
\end{lem}

\begin{proof}
We perform an induction on the length of the standard form.

\emph{Base case $n=1$} Consider a standard form $fw_{s,t}$.
Performing the amplification moves $$w_{s,t}\to x_sw_{s0,s10}w_{s11,t}$$ and $$w_{s11,t}\to x_t^{-1}w_{s11,t1}w_{t01,t00}$$
followed by commutation moves we obtain
$$(fx_sx_t^{-1}) (w_{s0,s10}w_{s11,t1}w_{t01,t00})$$
which has depth greater than the depth of the original standard form.
Now we can repeatedly perform such amplification moves on the offspring $w$'s,
followed by commutation moves, to obtain standard forms of the desired large depth.
Note that upon performing these moves systematically, we can ensure that the depth increases by a factor of at most one at each stage.

\emph{Inductive step}
Consider a standard form $fw_{s_1,t_1}^{l_1}...w_{s_n,t_n}^{l_n}$ and $m\in \mathbf{N}$.
Assume for the sake of convenience that $l_1=l_2=...=l_n=1$.
From the inductive hypothesis, we perform moves on the standard form $fw_{s_1,t_1}...w_{s_{n-1},t_{n-1}}$
to obtain a standard form which we denote (for simplicity) by $\xi$, whose depth is at least $m+3$.

Now using the argument in the base case, we perform a sequence of amplification moves on $w_{s_n,t_n}$ to obtain a standard form $g w_{\sigma_1,\tau_1}^{o_1}...w_{\sigma_k,\tau_k}^{o_k}$
of depth precisely $m$.
It is easy to see that the element $g\in F$ acts on each $w$-generator occurring to the left in the word $\xi$, thanks to the assumption on the depth of $\xi$. 
Hence, upon performing a sequence of rearrangement moves, we obtain a standard form of depth at least $m$.
\end{proof}

\begin{proof}
{\bf Proof of Lemma \ref{stdform}}:
We proceed by an induction on the \emph{word length} of a word $W$ in the infinite generating set $\mathcal{X}$.
The base case is trivial.

\emph{Inductive step} Let $W$ be a word of length $n$ of the form $W_1 l$, where $W_1$ is a word of length $n-1$ and $l$ is a generator in $\mathcal{X}$.
Using the inductive hypothesis, we convert $W_1$ to a standard form $\Omega_1$.
We have three cases:
\begin{enumerate}
\item $l=x_{\sigma}$ for some $\sigma\in 2^{<\mathbf{N}}$ or $\sigma=\emptyset$.
\item $l=p_T$ for some $T\in \mathcal{T}$.
\item $l=w_{\sigma,\tau}$ for some $\sigma,\tau$ that are independent binary sequences.
\end{enumerate}
In cases $(1)$ and $(2)$, we apply Lemma \ref{stddepth} to convert $\Omega_1$ to a standard form $\Omega_2$ of sufficiently large depth so that $l$ acts on
each $w$-generator of $\Omega_2$.
We then use the rearrangement relations to move $l$ to the left of all the $w$-generators (in particular, upon reconfiguring the subscripts of
the $w$-generators).
The resulting word is clearly a standard form.
In case $(3)$, we simply convert $\Omega_1$ to a standard form $\Omega_2$ of depth larger than $|\sigma|,|\tau|$,
so the resulting word $\Omega_2 l$ is a standard form.
\end{proof}

We now describe a move on standard forms.
This is a combination of moves described in \ref{relations}, and shall be denoted as the \emph{AR} move.
(AR stands for amplification followed by rearrangement).

Consider the standard form $\Omega=g w_{s_1,t_1}^{l_1}...w_{s_n,t_n}^{l_n}$
such that $x_{s_i}$ acts on each sequence in the set $\{s_j,t_j\mid j<i\}$.
We define the AR move \emph{performed on} $\Omega$ \emph{at} $w_{s_i,t_i}^{l_i}, s_i$ as follows:
\begin{enumerate}
\item Replace $w_{s_i,t_i}^{l_i}$ by $$x_{s_i}(w_{s_i0,s_i10}w_{s_i11,t_i}w_{s_i,t_i}^{(l_i-1)})$$ to obtain
$$g(w_{s_1,t_1}^{l_1}...w_{s_{i-1},t_{i-1}}^{l_{i-1}})(x_{s_i} w_{s_i0,s_i10}w_{s_i11,t_i}w_{s_i,t_i}^{(l_i-1)}) (w_{s_{i+1},t_{i+1}}^{l_{i+1}}...w_{s_{n},t_{n}}^{l_n})$$
\item Perform a rearrangement move to obtain 
$$gx_{s_i} (w_{s_1\cdot x_s,t_1\cdot x_s}^{l_1}...w_{s_{i-1}\cdot x_s,t_{i-1}\cdot x_s}^{l_{i-1}})( w_{s_i0,s_i10}w_{s_i11,t_i}w_{s_i,t_i}^{(l_i-1)}) (w_{s_{i+1},t_{i+1}}^{l_{i+1}}...w_{s_{n},t_{n}}^{l_n})$$
\end{enumerate}
The output of the move is the resulting word.\\

Similarly, define the AR move performed on $\Omega$ at $w_{s_i,t_i}^{l_i}, t_i$ as follows:
\begin{enumerate}
\item Replace $w_{s_i,t_i}^{l_i}$ by $$x_{t_i}^{-1}(w_{s_i,t_i1}w_{t_i01,t_i00}w_{s_i,t_i}^{(l_i-1)})$$ to obtain
$$g(w_{s_1,t_1}^{l_1}...w_{s_{i-1},t_{i-1}}^{l_{i-1}})(x_{t_i}^{-1}(w_{s_i,t_i1}w_{t_i01,t_i00}w_{s_i,t_i}^{(l_i-1)}) (w_{s_{i+1},t_{i+1}}^{l_{i+1}}...w_{s_{n},t_{n}}^{l_n})$$
\item Perform a rearrangement move to obtain 
$$gx_{t_i}^{-1} (w_{s_1\cdot x_{t_i}^{-1},t_1\cdot x_{t_i}^{-1}}^{l_1}...w_{s_{i-1}\cdot x_{t_i}^{-1},t_{i-1}\cdot x_{t_i}^{-1}}^{l_{i-1}})(w_{s_i,t_i1}w_{t_i01,t_i00}w_{s_i,t_i}^{(l_i-1)}) (w_{s_{i+1},t_{i+1}}^{l_{i+1}}...w_{s_{n},t_{n}}^{l_n})$$
\end{enumerate}

The following is an elementary observation and the proof is left to the reader.

\begin{lem} 
Let $\Gamma$ be a word obtained by performing an AR move on a standard form $\Omega$.
Then $\Gamma$ is a standard form.
\end{lem}

The set of elements of $S$ that fix $\infty=0^{\infty}=1^{\infty}$ form a subgroup we denote by $S_{\infty}$.
The following is a basic observation about $S_{\infty}$.

\begin{lem}\label{ginF}
If $gw_{s_1,t_1}^{l_1}...w_{s_n,t_n}^{l_n}$ is a standard form representation for an element of $S_{\infty}$,
then $g\in F$.
\end{lem}

\begin{proof}
Note that the homeomorphism represented by the word $w_{s_1,t_1}^{l_1}...w_{s_n,t_n}^{l_n}$ fixes $\infty$.
Therefore if $gw_{s_1,t_1}^{l_1}...w_{s_n,t_n}^{l_n}$ fixes $\infty$, then $g$ must fix $\infty$.
Since the subgroup of $T$ consisting of elements that fix infinity is precisely $F$, we obtain the desired conclusion. 
\end{proof}

We denote the set of words in standard form representing elements of $S,S_{\infty},G$ respectively by $\mathcal{S},\mathcal{S}_{\infty},\mathcal{G}$.
Given a standard form $$fw_{s_1,t_1}^{l_1}...w_{s_n,t_n}^{l_n}\in \mathcal{S}_{\infty}$$
the homeomorphism it describes is representable as a standard form in $\mathcal{G}$ as
$$f y_{s_1}^{l_1} y_{t_1}^{l_1}...y_{s_n}^{l_n}y_{t_n}^{l_n}$$
This is a standard form,
since for each $i<j$ and each pair $u\in \{s_i,t_i\},v\in \{s_j,t_j\}$ it holds that either $u,v$ are independent or $v\subset u$.
The latter standard form shall be denoted as a \emph{literal translation} of the former.
To make matters precise, we define the literal translation partial map $$\Lambda:\mathcal{S}_{\infty}\to \mathcal{G}$$
which literally translates a standard form in $\mathcal{S}_{\infty}$ to a standard form in $\mathcal{G}$.

The element $w_{s,t}$ is said to be \emph{balanced}, if either 
$1\subset s, 1\subset t$ or $0\subset s, 0\subset t$. 
Otherwise, it is said to be \emph{unbalanced}.
The two situations for unbalanced $w_{s,t}$ shall be denoted as \emph{parities}
of $w_{s,t}$.
For instance, $w_{10,01}, w_{110,001}$ have the same parity, and $w_{10,01}, w_{001,110}$ have different parities.

A standard form $\Omega=gw_{s_1,t_1}^{l_1}...w_{s_n,t_n}^{l_n}$ is said to be \emph{balanced}, if for each $1\leq i\leq n$,
either $1\subset s_i, 1\subset t_i$ or $0\subset s_i, 0\subset t_i$. 
If the standard form is not balanced, then we denote the quantity $$\Xi(\Omega)=\sum_{i\in J} l_i$$ 
$$\text{ where }J=\{1\leq i\leq n\mid 1\subset s_i, 0\subset t_i\text{ or } 0\subset s_i,1\subset t_i\}$$
as the \emph{unevenness index}.
In particular, the standard form is balanced if the unevenness index is zero.

We end this subsection by defining a few notions concerning standard forms which shall be useful in the proof
of the sufficiency of the relations.
In order to keep the notation simple, we shall often abuse notation and confuse the subscripts of the $w$-generators
as elements themselves. 
But the usage will always be clear from the context.

Consider a standard form $\Omega=gw_{s_1,t_1}^{l_1}...w_{s_n,t_n}^{l_n}$.
For $i<j$, the pair $s_i,s_j$ is said to be an \emph{adjacent pair} in $\Omega$ if:
\begin{enumerate}
\item $s_j\subset s_i$.
\item For $s\in 2^{<\mathbf{N}}$ such that $s_j\subset s\subset s_i$, it holds that $s\notin \{s_1,t_1,s_2,t_2,...,s_n,t_n\}$.
\end{enumerate}
The notion is defined similarly for pairs of the form $$s_i,t_j\qquad t_i,t_j\qquad t_i,s_j$$

The sequence $s_i$ (or $t_i$) is said to be \emph{sheltered} in $\Omega$ if for any infinite binary sequence $\psi$
that contains $s_i$ (respectively $t_i$) as an initial segment, there is a sequence $$u\in \{s_1,t_1,...,s_{i-1},t_{i-1}\}$$
such that $s_i\subset u\subset \psi$ (respectively $t_i\subset u\subset \psi$).
The sequence $s_i$ (respectively $t_i$) is said to be \emph{exposed} in $\Omega$ if it is not \emph{sheltered}.

The sequence $s_i$ is said to be \emph{free} in $\Omega$ if $x_{s_i}$ acts on each sequence in the set $\{s_1,t_1,...,s_{i-1},t_{i-1}\}$.
If $s_i$ is not free, then there sequence $u\in \{s_1,t_1,...,s_{i-1},t_{i-1}\}$ such that $u=s_i0$, and $u$ is said to be a \emph{barrier} for $s_i$.
Similarly, the sequence $t_i$ is free in $\Omega$ if $x_{s_i}^{-1}$ acts on each sequence in the set $\{s_1,t_1,...,s_{i-1},t_{i-1}\}$.
The notion of barrier here is analogously defined.
Note that if $s_i$ is free in $\Omega$, then it is possible to perform an AR move on $\Omega$ at $w_{s_i,t_i}^{l_i}, s_i$.
Similarly, if $t_i$ is free in $\Omega$, then it is possible to perform an AR move on $\Omega$ at $w_{s_i,t_i}^{l_i}, t_i$.

The standard form $\Omega$ is said to be \emph{expansible}, if there is an $s_i$ (or $t_i$) such that $s_i$ (respectively $t_i$) is simultaneously sheltered and free in $\Omega$.

\subsection{Sufficiency of the relations and finite presentability}

Now we establish the following:
\begin{prop}\label{infinitepresentation}
$\langle \mathcal{X}, \mathcal{R}\rangle\cong S$.
\end{prop}

The proof of this proposition shall be described in the following steps:
\begin{enumerate}
\item[Step 1] We describe a process to convert a standard form into a balanced standard form using the relations in $\mathcal{R}$.
\item[Step 2] Given a balanced standard form that represents the identity element, we use the relations in $\mathcal{R}$ to convert it to a word in the $x$-generators.
\end{enumerate}

We perform the steps in the order prescribed, although the reader may wish to read them in the opposite order, 
to motivate the need for balanced standard forms.

{\bf Step 1}

First we show that performing AR or cancellation moves cannot increase the unevenness index.
\begin{lem}\label{balancing1}
Let $\Gamma$ be a standard form obtained from applying a sequence of AR and cancellation moves on a standard form $\Omega$.
Then the unevenness index of $\Gamma$ is at most the unevenness index of $\Omega$. 
\end{lem}

\begin{proof}
Let $\Omega=g w_{s_1,t_1}^{l_1}...w_{s_n,t_n}^{l_n}$
such that $s_i$ is free in $\Omega$.
Consider an AR move on $\Omega$ at $w_{s_i,t_i}^{l_i}, s_i$.
\begin{enumerate}
\item Replace $w_{s_i,t_i}^{l_i}$ by $$x_{s_i}(w_{s_i0,s_i10}w_{s_i11,t_i}w_{s_i,t_i}^{(l_i-1)})$$ to obtain
$$g(w_{s_1,t_1}^{l_1}...w_{s_{i-1},t_{i-1}}^{l_{i-1}})(x_{s_i} w_{s_i0,s_i10}w_{s_i11,t_i}w_{s_i,t_i}^{(l_i-1)}) (w_{s_{i+1},t_{i+1}}^{l_{i+1}}...w_{s_{n},t_{n}}^{l_n})$$
\item Perform a rearrangement move to obtain 
$$gx_{s_i} (w_{s_1\cdot x_s,t_1\cdot x_s}^{l_1}...w_{s_{i-1}\cdot x_s,t_{i-1}\cdot x_s}^{l_{i-1}})( w_{s_i0,s_i10}w_{s_i11,t_i}w_{s_i,t_i}^{(l_i-1)}) (w_{s_{i+1},t_{i+1}}^{l_{i+1}}...w_{s_{n},t_{n}}^{l_n})$$
\end{enumerate}
In this first stage, if $w_{s_i,t_i}$ is balanced, then so are the offsprings $w_{s_i0,s_i10},w_{s_i11,t_i}$.
If $w_{s_i,t_i}$ is unbalanced, then note that $w_{s_i11,t_i}$ will also be unbalanced, but $w_{s_i0,s_i10}$ will be balanced, since $|s|\geq 1$.
In the second stage, note that since $|s|\geq 1$, the partial action of $x_s$ preserves the subtrees rooted at $1$ and $0$ respectively.
It follows that the unevenness index of 
$$w_{s_1\cdot x_s,t_1\cdot x_s}^{l_1}...w_{s_{i-1}\cdot x_s,t_{i-1}\cdot x_s}^{l_{i-1}}$$
is the same as that of
$$w_{s_1,t_1}^{l_1}...w_{s_{i-1},t_{i-1}}^{l_{i-1}}$$
The proof for an AR move of the type $w_{s_i,t_i}^{l_i}, t_i$ is similar.

Consider a cancellation move $w_{\sigma,\tau}w_{\tau,\nu}\to w_{\sigma,\nu}$ where $\sigma,\nu$ are independent.
Note that if $w_{\sigma,\nu}$ is unbalanced, then it must be the case that either $w_{\sigma,\tau}$ or $w_{\tau,\nu}$ are unbalanced.
Our conclusion follows.
\end{proof}

\begin{lem}\label{balancing2}
Consider a standard form $w_{u,v} \Gamma_1$ such that $w_{u,v}$ is unbalanced and $\Gamma_1$ is balanced.
Then we can apply a sequence of moves on $w_{u,v} \Gamma_1$ to obtain an equivalent standard form $g \Gamma_2 w_{u_1,v_1}$ such that:
\begin{enumerate}
\item $g\in F$ and $\Gamma_2$ is a balanced standard form.
\item $w_{u_1,v_1}$ is unbalanced and has the same parity as $w_{u,v}$.
\end{enumerate}
\end{lem}

\begin{proof}
We shall assume in this proof that $0\subset u,1\subset v$.
The case with $1\subset u,0\subset v$ is symmetric.
We proceed by induction on the word length of $\Gamma_1$.

\emph{ Base Case}:  $\Gamma_1=w_{s,t}$.
We know that $w_{s,t}$ is balanced.
The two possibilites
$$0\subset s, 0\subset t\qquad 1\subset s, 1\subset 1$$
are symmetric.
So we shall assume the former, i.e. $0\subset s, 0\subset t$.

Now there are three cases we need to consider:
\begin{enumerate}
\item $s,t,u$ are independent.
\item $s\subset u$.
\item $t\subset u$.
\end{enumerate}

{\bf Case $(1)$}: In this case we can apply the commutation move $$w_{u,v}w_{s,t}\to w_{s,t}w_{u,v}$$ since $u,v,s,t$ are independent.

{\bf Case $(2)$}: First we apply the amplification move $$w_{s,t}\to x_t^{-1} w_{s,t1}w_{t01,t00}$$
followed by the commutation move $$w_{s,t1}w_{t01,t00}\to w_{t01,t00} w_{s,t1}$$
to obtain $$w_{u,v} (x_t^{-1} (w_{t01,t00} w_{s,t1}))$$
Now since $$u\cdot x_t^{-1}=u\qquad v\cdot x_t^{-1}=v$$
we perform a rearrangement move to obtain 
$$x_t^{-1} (w_{u,v} w_{t01,t00} w_{s,t1})$$
By our assumptions, $u,v,t01,t00$ are independent.
So we perform the relabelling move
$$w_{u,v} w_{t01,t00}\to w_{u,t00} w_{t01,v}$$
to obtain $$x_t^{-1}( w_{u,t00} w_{t01,v} w_{s,t1}) $$

Finally, since $t01,v,s,t1$ are independent, we can apply a commutation move 
$$w_{t01,v} w_{s,t1}\to w_{s,t1}w_{t01,v}$$
to obtain $$x_t^{-1} (w_{u,t00}w_{s,t1}w_{t01,v})$$
Since $w_{u,t00},w_{s,t1}$ are balanced and $w_{t01,v}$ is unbalanced, we are done.

{\bf Case $(3)$}: This is similar to case $(2)$. We proceed in a similar fashion, with the exception that we perform the amplification move 
at $w_{s,t},s$ instead of $w_{s,t},t$.

{\bf Inductive step}: Now we assume the conclusion holds for $n$, and we wish the consider the case $|\Gamma_1|=n+1$.
Let $\Gamma_1=\Lambda w_{s,t}$ where $\Lambda$ is a balanced standard form of length $n$.
Using the inductive hypothesis, we apply moves to convert the word $w_{u,v}\Lambda$ to a standard form $\Lambda_1 w_{u',v'}$
where $\Lambda_1$ is balanced and $w_{u',v'}$ is unbalanced and has the same parity as $w_{u,v}$.
Next, we apply the argument in the base case to convert the standard form $w_{u',v'} w_{s,t}$ to a standard form $\Lambda_2 w_{u_1,v_1}$
where $\Lambda_2$ is balanced and $w_{u_1,v_1}$ is unbalanced and has the same parity as $w_{u',v'}$, and hence as $w_{u,v}$.
So we obtain the word $$\Lambda_1 \Lambda_2 w_{u_1,v_1}$$
Now let $$\Lambda_2= f w_{\psi_1,\tau_1}^{k_1}...w_{\psi_m,\tau_m}^{k_m}$$ where $f\in F$.
(This is simply an equality as words).
Next, using a sequence of moves, we convert $\Lambda_1$ to a balanced standard form $$g w_{s_1,t_1}^{l_1}...w_{s_n,t_n}^{l_n}$$
of sufficiently large depth such that:
\begin{enumerate}
\item $f$ acts on each each element of the set $\{s_1,t_1,...,s_n,t_n\}$.
\item $|s_i\cdot f|\geq \sum_{1\leq i\leq m}(|\psi_i|+|\tau_i|) + |u_1|+|v_1|$.
\end{enumerate}
It follows that $$gf (w_{s_1\cdot f,t_1\cdot f}^{l_1}...w_{s_n\cdot f,t_n\cdot f}^{l_n}) (w_{\psi_1,\tau_1}^{k_1}...w_{\psi_m,\tau_m}^{k_m})( w_{u_1,v_1})$$
is the required standard form since:
\begin{enumerate}
\item $$gf (w_{s_1\cdot f,t_1\cdot f}^{l_1}...w_{s_n\cdot f,t_n\cdot f}^{l_n}) (w_{\psi_1,\tau_1}^{k_1}...w_{\psi_m,\tau_m}^{k_m})$$
is balanced.
\item $w_{u_1,v_1}$ is unbalanced and has the same parity as $w_{u,v}$.
\end{enumerate}
\end{proof}

\begin{lem}\label{balancing3}
Consider a standard form $w_{s_1,t_1} w_{s_2,t_2}$ such that:
\begin{enumerate}
\item $|s_1|,|s_2|,|t_1|,|t_2|\geq 2$.
\item $w_{s_1,t_1}$ and $w_{s_2,t_2}$ are unbalanced and have different parities.
\end{enumerate}
Then we can perform a sequence of moves on $w_{s_1,t_1} w_{s_2,t_2}$ to obtain a standard form $\Omega$ which is balanced.
\end{lem}

\begin{proof}
It suffices to show this for the case $0\subset s_1, 1\subset t_1$ and hence $1\subset s_2, 0\subset t_2$.
The other case is symmetric.

There are four possible subcases:

\begin{enumerate}
\item $s_1,s_2,t_1,t_2$ are independent.
\item $t_2\subset s_1$ and $s_2\subset t_1$. 
\item $t_2\subset s_1$ and $s_2,t_1$ are independent.
\item $t_2,s_1$ are independent and $s_2\subset t_1$.
\end{enumerate}

{\bf Case $(1)$}: In this case we simply apply a relabelling move $$w_{s_1,t_1} w_{s_2,t_2}\to w_{s_1,t_2} w_{s_2,t_1}$$
to obtain a balanced standard form.

{\bf Case $(2)$}: Our goal will be to reduce this to Case $(3)$.

First, upon performing amplification moves at $w_{s_1,t_1},t_1$, and offsprings, we obtain a standard form $\Gamma w_{s_1,u}$ such that:
\begin{enumerate}
\item $\Gamma$ is balanced.
\item $w_{s_1,u}$ is unbalanced and has the same parity as $w_{s_1,t_1}$.
\item The depth of $\Gamma$ is at least $|s_2|+6$.
\item $|u|\geq |s_2|+6$.
\end{enumerate}

Now apply amplification moves $$w_{s_2,t_2}\to x_{s_2} w_{s_20,s_210}w_{s_211, t_2}$$
followed by $$w_{s_20,s_210}\to x_{s_20}w_{s_200,s_2010}w_{s_2011,s_210}$$ to obtain 
{\large $$(x_{s_2}x_{s_20})(w_{s_200,s_2010}w_{s_2011,s_210}w_{s_211, t_2})$$}
Substituting this in the original word we obtain:
{\large$$\Gamma w_{s_1,u}(x_{s_2}x_{s_20}) (w_{s_200,s_2010}w_{s_2011,s_210}w_{s_211, t_2})$$}
Let $g= x_{s_2}x_{s_20}$.
By the assumption on depth, we can apply rearrangement moves to obtain a standard form of the type:
{\large$$\Gamma_1 w_{s_1,u\cdot g} (w_{s_200,s_2010}w_{s_2011,s_210}w_{s_211, t_2})$$}
where $\Gamma_1$ is balanced.
Since {\large $$s_1, s_200,s_2010,s_2011,s_210$$} are independent and we know that $u\cdot g$ is longer than all these sequences, either one of the following situations holds:
\begin{enumerate}
\item $u\cdot g, s_1, s_200,s_2010$ are independent.
\item $u\cdot g, s_1,s_2011,s_210$ are independent.
\end{enumerate}
We assume the former, the argument in the latter case is similar since {\large $$w_{s_200,s_2010}w_{s_2011,s_210}=w_{s_2011,s_210}w_{s_200,s_2010}$$}

We perform the relabelling move
{\large$$w_{s_1,u\cdot g} w_{s_200,s_2010}\to w_{s_1,s_2010}w_{s_200,u\cdot g}$$}
followed by the commutation move {\large$$w_{s_1,s_2010}w_{s_200,u\cdot g}\to w_{s_200,u\cdot g}w_{s_1,s_2010}$$}
to obtain 
{\large$$\Gamma_1  (w_{s_200,u\cdot g} w_{s_1, s_2010}) (w_{s_2011,s_210} w_{s_211, t_2})$$}
and one more commutation move 

{\large $$w_{s_1, s_2010} w_{s_2011,s_210}\to w_{s_2011,s_210} w_{s_1, s_2010}$$}

to obtain 

{\large$$\Gamma_1  (w_{s_200,u\cdot g} w_{s_2011,s_210}) (w_{s_1, s_2010}w_{s_211, t_2}) $$}

Note that {\large$$\Gamma_1 (w_{s_200,u\cdot g} w_{s_2011,s_210})$$} is a balanced standard form,
and {\large $$(w_{s_1, s_2010}w_{s_211, t_2}) $$} falls into Case $(3)$.
We convert $(w_{s_1, s_2010}w_{s_211, t_2}) $ into a balanced standard form $h \Gamma_2$ using case $(3)$,
where $h\in F$ and $\Gamma_2$ is a balanced $w$-standard form.
Next, we convert {\large$$\Gamma_1 w_{s_200,u\cdot g} w_{s_2011,s_210}$$}
to a balanced standard form $h_1 (w_{u_1,v_1}^{l_1}...w_{u_n,v_n}^{l_n})$
of sufficiently large depth so that $h$ acts on each element of the set $\{u_1,v_1,...,u_n,v_n\}$
and so that upon performing a rearrangement move we obtain a balanced standard form
{\large$$h_1 h (w_{u_1\cdot h,v_1\cdot h}^{l_1}...w_{u_n\cdot h,v_n\cdot h}^{l_n})  \Gamma_2$$}

{\bf Case $(3)$}: First, upon performing amplification moves on $w_{s_1,t_1},s_1$,
we obtain a standard form $\Gamma w_{v,t_1}$ such that:
\begin{enumerate}
\item $\Gamma$ is balanced and $w_{v,t_1}$ is unbalanced and has the same parity as $w_{s_1,t_1}$.
\item The depth of $\Gamma$ is at least $|t_2|+6$.
\item $|v|\geq |t_2|+6$.
\end{enumerate}
Moreover, the natural sequence of amplifications that provides the above also ensures that $t_2\subset v$.

Now we apply the amplification moves
{\large $$w_{s_2,t_2}\to x_{t_2}^{-1} w_{s_2,t_21} w_{t_201,t_200}$$}
followed by {\large $$w_{s_2,t_21}\to x_{t_21}^{-1} w_{s_2,t_211} w_{t_2101,t_2100}$$}
to obtain {\large $$x_{t_2}^{-1} x_{t_21}^{-1} (w_{s_2,t_211} w_{t_2101,t_2100} w_{t_201,t_200})$$}
Next we apply commutation moves
{\large $$w_{s_2,t_211} w_{t_2101,t_2100}\to w_{t_2101,t_2100}w_{s_2,t_211} $$}
and {\large $$w_{s_2,t_211}w_{t_201,t_200}\to w_{t_201,t_200}  w_{s_2,t_211}$$ }
to obtain 
{\large $$x_{t_2}^{-1}x_{t_21}^{-1}  (w_{t_2101,t_2100}w_{t_201,t_200}w_{s_2,t_211} )$$}
Upon substituting in the original standard form we obtain:
{\large $$\Gamma w_{ v, t_1} g (w_{t_2101,t_2100}w_{t_201,t_200}w_{s_2,t_211} )$$}
where $g=x_{t_2}^{-1}x_{t_21}^{-1}$.

By our assumption on depth, $g$ acts on $v$ and each $w$-word of $\Gamma$.
Upon performing rearrangement moves, we obtain a standard form of the type
{\large $$\Gamma_1 w_{ v\cdot g, t_1} (w_{t_2101,t_2100}w_{t_201,t_200}w_{s_2,t_211} )$$}
Either one of the following situations holds:
\begin{enumerate}
\item $v\cdot g, t_1, t_2101,t_2100$ are independent.
\item $v\cdot g, t_1, t_201,t_200$ are independent.
\end{enumerate}
We assume the former, the argument in the latter case is similar.

We perform the relabelling move
{\large $$w_{v\cdot g, t_1} w_{t_2101,t_2100}\to w_{v\cdot g,t_2100}w_{t_2101,t_1}$$}

to obtain 

{\large $$\Gamma_1  (w_{v\cdot g,t_2100} w_{t_2101,t_1}) ( w_{t_201,t_200}w_{s_2,t_211} ) $$}

Another commutation move

{\large $$w_{t_2101,t_1} w_{t_201,t_200}\to w_{t_201,t_200}w_{t_2101,t_1}$$}
give us 

{\large $$\Gamma_1 (w_{v\cdot g,t_2100}  w_{t_201,t_200}) (w_{t_2101,t_1} w_{s_2,t_211} )  $$}

Finally, observe that $$t_2101,t_1, s_2,t_211$$ are independent.
Then using Case $(1)$ we are done.

{\bf Case $(4)$}: This is symmetric to Case $(3)$.

\end{proof}

\begin{prop}\label{balancing}
Let $\Omega$ be a standard form in $\mathcal{S}$. We can perform a sequence of moves to obtain an equivalent standard form $\Gamma$ which is balanced. 
\end{prop}

\begin{proof}
Let $$\Omega=\Gamma_1 w_{s_1,t_1} \Gamma_2 w_{s_2,t_2} \Gamma_3$$ 
where equality denotes equality as a word, such that:
\begin{enumerate}
\item $\Gamma_2$ is a balanced standard form.
\item $w_{s_1,t_1}$ and $w_{s_2,t_2}$ are unbalanced and have opposite parities.
\end{enumerate}
By an elementary parity argument, one can always find such an expression as above for $\Omega$.
First, we apply a sequence of moves provided by Lemma \ref{balancing2} to convert $w_{s_1,t_1} \Gamma_2$ into a standard form $\Gamma_4 w_{s_3,t_3}$
such that $\Gamma_4$ is balanced and $w_{s_3,t_3},w_{s_1,t_1}$ have the same parity.
Then we use moves provided by Lemma \ref{balancing3} to convert $w_{s_3,t_3}w_{s_2,t_2}$ to a balanced standard form $\Gamma_5$.
The sum of the unevenness indices of $\Gamma_1, \Gamma_4, \Gamma_5, \Gamma_3$ is smaller than the unevenness index of $\Omega$.
Note that the resulting word, which may not be a standard form, is:
$$\Gamma_1  \Gamma_4 \Gamma_5 \Gamma_3$$

Next, convert the standard form $\Gamma_5$ to a standard form $\Gamma_6$ with sufficiently large depth so that $\Gamma_6\Gamma_3$ is a standard form.
Let $\Gamma_6\Gamma_3=f \Lambda$ where $f\in F$ and $\Lambda$ is a $w$-word in standard form.
Finally, convert $\Gamma_1\Gamma_4$ to a standard form $gw_{s_1,t_1}^{l_1}...w_{s_n,t_n}^{l_n}$ with sufficiently large depth so that:
\begin{enumerate}
\item $f$ acts on each element of the set $\{s_1,t_1,...,s_n,t_n\}$.
\item For each $w$-word $w_{s,t}$ in $\Lambda$, we have that $$|s_i\cdot f|>|s|+|t|\qquad |t_i\cdot f|>|s|+|t|$$
\end{enumerate}
In both steps, we obtain standard forms with unevenness index at most that of the standard form we begin with, thanks to Lemma \ref{balancing1}.
The resulting word $$gf (w_{s_1\cdot f,t_1\cdot f}^{l_1}...w_{s_n\cdot f,t_n\cdot f}^{l_n})\Lambda$$ 
is a standard form that is equivalent to the original standard form, but has smaller unevenness index.
Continuing in this fashion, we apply a sequence of moves to obtain a balanced standard form.
\end{proof}

{\bf Step 2}

Consider a balanced standard form $\Omega=gw_{s_1,t_1}^{l_1}...w_{s_n,t_n}^{l_n}$.
Using commutation moves, we can convert this into a standard form $g\Omega_0\Omega_1$
such that: 
\begin{enumerate}
\item For each $w_{\sigma,\tau}$ occurring in $\Omega_0$, we have $0\subset \sigma, 0\subset \tau$. 
\item For each $w_{\sigma,\tau}$ occurring in $\Omega_1$, we have $1\subset \sigma, 1\subset \tau$. 
\end{enumerate}
We assume for the rest of this step that $\Omega=1_S$.

\begin{lem}
$\Omega_0\in F$ and $\Omega_1\in F$.
\end{lem}

\begin{proof}
Since $\Omega=1_S$, by Lemma \ref{ginF} we know that $g\in F$, and hence it follows that $\Omega_0\Omega_1\in F$.
Since the action of $\Omega_0\Omega_1$ on the subtrees rooted at $0,1$ respectively is the same as that of $\Omega_0,\Omega_1$ respectively,
our conclusion holds.
\end{proof}

\begin{prop}\label{balancedreduction}
The words $\Omega_0$ and $\Omega_1$ can be reduced to a word in the $x$-generators by a sequence of moves.
\end{prop}

\begin{proof}

We shall prove the above proposition for $\Omega_0$.
The proof for $\Omega_1$ is similar.
We will show that the word $$\Omega_0 w_{1,1}^n\qquad \text{ where }n=|\Omega_0|$$
is reducible to the empty word using a sequence of moves.

First observe that $w_{1,1}$ commutes with each generator $w_{\sigma,\tau}$ in $\Omega_0$.
Using commutation and relabelling moves, we convert the word $\Omega_0 w_{1,1}^n$
to a standard form $$\Lambda= w_{\sigma_1,\tau_1}^{k_1}...w_{\sigma_m,\tau_m}^{k_m}$$
satisfying that:

\begin{enumerate}
\item[(P)] Either $\Lambda$ contains no occurrences of $w$-generators, or the $w$-part $$w_{\sigma_1,\tau_1}^{k_1}...w_{\sigma_m,\tau_m}^{k_m}$$ satisfies the following for each $1\leq i\leq m$:
\begin{enumerate}
\item[(1)] Either $\sigma_i=1$ or $\tau_i=1$.
\item[(2)] If $u\in \{\sigma_i,\tau_i\}$ satisfies that $u\neq 1$, then $0\subset u$.
\end{enumerate}
\end{enumerate}

We now describe a \emph{reduction procedure} that takes an input the standard form $\Lambda$,
and perform the following manipulations, one by one, until no such manipulation can be performed:

\begin{enumerate}

\item Find a pair of generators $w_{\sigma,1}, w_{1,\sigma}$.
Apply a sequence of commutation moves to obtain a subword $w_{\sigma,1}w_{1,\sigma}$.
Perform the cancellation moves $$w_{\sigma,1}w_{1,\sigma}\to w_{1,1}\to \emptyset$$

\item If a pair of generators as above do not exist, 
find a $w_{\sigma,\tau}$ such that $\sigma$ (or $\tau$) is sheltered and free in the standard form.
Perform the AR move $w_{\sigma,\tau},\sigma$ (respectively $w_{\sigma,\tau},\tau$).

\item Apply the above procedure (again starting at Step $1$) to the output of the previous step.
\end{enumerate}

The reduction procedure involves performing the above manipulation to a standard form satisfying $(P)$,
one by one, until no such manipulation can be performed.
We will now show that the procedure terminates and that each of the steps is well defined.

\begin{lem}
Both manipulations in the reduction procedure produce a new standard form satisfying $(\textup{P})$.
\end{lem}

\begin{proof}
This is elementary and left to the reader.
\end{proof}

\begin{lem}
(Cancellation) If a pair of generators $w_{\sigma,1}, w_{1,\sigma}$ exist in a standard form that satisfies $(\textup{P})$,
then the second manipulation can be performed.
That is, we can apply a sequence of commutation moves to obtain a subword $w_{\sigma,1}w_{1,\sigma}$,
and then perform the cancellation.
\end{lem}

\begin{proof}
By the definition of a standard form, if $w_{s,t}$ occurs between the occurrences of $w_{\sigma,1},w_{1,\sigma}$
in the standard form, then for each $u\in \{s,t\}$ we conclude that $u,\sigma$ and $u,1$ are independent.
It follows that we can apply the desired commutation moves.
\end{proof}

If none of the above moves cannot be performed, we say that the standard form is \emph{fully developed}.

\begin{lem}\label{processstops}
Let $\Gamma$ be a standard form that satisfies $(\textup{P})$.
Then upon performing (in any order) the above manipulations, one by one, until no manipulation can be performed,
results in a standard form which is fully developed. 
\end{lem}

\begin{proof}
Before we prove this statement, we digress to discuss the following scenario that emerges at the technical core of the proof.

{\bf Technical core scenario}: Consider a set $\{s_1,...,s_n\}$ of finite binary sequences with the property that $$\textup{max}\{|s_i|\mid 1\leq i\leq n\}<m$$
for some $m\in \mathbf{N}$. 
We describe two \emph{moves} on such a set:
\begin{enumerate}
\item[play 1] Discard one or more elements from the set. 
\item[play 2] Replace an element $s_i$ of the set by three finite binary sequences, each of length greater than $|s_i|$.
\end{enumerate}
A move is said to be \emph{admissible} if the starting set is nonempty, and each sequence in the resulting set has length at most $m$.\\

{\bf Observation}: There is a finite upper bound on the number of successive admissible moves that can be performed on the set $\{s_1,...,s_n\}$ that only depends on $n,m$.
The proof of this observation is elementary and left to the reader.
An upper bound is $2\cdot n\cdot 3^{m+1}$, for instance.

Now we return to the proof of our Lemma \ref{processstops}.
Let $$\Lambda^{(1)}=\Lambda\to \Lambda^{(2)}\to...\to\Lambda^{(m)}\to...$$
be the sequence of standard forms that result from performing the sequence of steps.
First observe that the set of subscripts $K_i$ of $\Lambda^{(i)}$, for each $i$, satisfies the following. 
The length of each sequence in $K_i$ is bounded above by the number of leaves of the finite rooted binary tree $T$
that contains all elements of the set $K_1$ as nodes.
To see this, observe that when we perform the first AR move in the procedure, the element of $F$ which is involved in the rearrangement
acts on the set of leaves of $T$. 
Hence the resulting set has the same cardinality as the number of leaves of $T$.
The same holds for each subsequent AR move.
Our claim follows.
Finally, observe that each manipulation corresponds to a ``technical core scenario move" $K_i\to K_{i+1}$.
Since such a process must end after finitely many steps, we conclude the proof.
\end{proof}

\begin{lem}
Let $\Gamma$ be a standard form that is fully developed.
Then either of the following holds:
\begin{enumerate}
\item $\Gamma$ does not contain any occurrences of $w$-generators.
\item $\Gamma\notin F$.
\end{enumerate}
\end{lem}

\begin{proof}
Assume that $\Gamma$ contains occurrences of $w$-generators, and is fully developed.
Let $\Gamma=g w_{s_1,t_1}^{l_1}...w_{s_n,t_n}^{l_n}$.
Let $s_i$ be a sequence with the property that no proper initial segment of $s_i$ belongs to the set $\{s_1,t_1,...,s_n,t_n\}$.
Either $s_i$ is exposed, or it is sheltered.
If it is sheltered, then is not free, since no AR moves can be performed.
Let $s_j$ (or $t_j$) be the barrier to $s_i$.
Continuing the same line of argument, we obtain a sequence $v_1=s_i,v_2,...,v_k$ of finite binary sequences in the set $\{s_1,t_1,...,s_n,t_n\}$,
each of which is a barrier to the previous one, and the last one is exposed.

It follows that the action of $\Gamma$ on the interval $[v_k0^{\infty},v_k 1^{\infty}]$ agrees with the action of a standard form $y_{v_k}^{o_1}...y_{v_1}^{o_1}\in \mathcal{G}$
(for some $o_1,...,o_k\in \mathbb{Z}\setminus \{0\}$) that does not contain potential cancellations.
By Lemma \ref{taileq}, the action of $y_{v_k}^{o_1}...y_{v_1}^{o_1}$ does not preserve tail equivalence on a dense subset of the interval $[v_k0^{\infty},v_k 1^{\infty}]$.
It follows that the action of $\Gamma$ does not preserve tail equivalence on a dense subset of the interval $[v_k0^{\infty},v_k 1^{\infty}]$.
It follows that $\Gamma\notin F$.
\end{proof}

{\bf Proof of Proposition \ref{infinitepresentation}}:

Consider a word $W$ in the generating set $\mathcal{X}$ that represents the identity element.
It suffices to show that using the relations in $\mathcal{R}$ we can convert it to the empty word.
First, we convert this word into a word in standard form using Lemma \ref{stdform}.
Using Step $1$ (in particular, Proposition \ref{balancing}), we convert the standard form into a standard form $g \Omega_0\Omega_1$
which is balanced.
Thanks to Proposition \ref{balancedreduction}, the standard forms $\Omega_0,\Omega_1$ can be reduced to elements $g_1,g_2\in F$
using the relations.
In effect, we have reduced the original word $W$ to a word $gg_1g_2$ in the generators of $F$.
Since the relations in $\mathcal{R}$ suffice to reduce this word to an empty word, we are done.
\end{proof}

\subsection{Reduction to a finite presentation}
In this subsection our goal shall be to prove the following:

\begin{thm}\label{main}
The group $S$ is finitely presented.
\end{thm}

\begin{proof}
We shall prove that $$\langle \mathcal{X}\mid \mathcal{R}\rangle\cong \langle \mathcal{X}_1\mid \mathcal{R}_1\rangle$$
where $$\mathcal{X}_1=\{x, x_1, p_0, w_{10,110}, w_{10,1110}\}$$
and $\mathcal{R}_1$ is a finite set of words in $\mathcal{X}_1$.
Before we describe $\mathcal{R}_1$, we first fix a reduced word in $$\{x,x_1,p_0,w_{10,110}, w_{10,1110}\}$$ for each generator
in the set $$\{x_{\sigma}, p_n\mid \sigma,\tau\in 2^{<\mathbb{N}}, n\in \mathbb{N}\}$$
Next, for any pair $s,t$ of independent finite binary sequences we fix an element $A_{s,t}$ in $T$ such that
either one of the following holds:
\begin{enumerate}
\item $A_{s,t}(10)=s, A_{s,t}(110)=t$.
\item $A_{s,t}(10)=s, A_{s,t}(1110)=t$.
\end{enumerate}
Finally, for any such pair $s,t$, depending on the case, we fix the word $$A_{s,t}^{-1} w_{10,110} A_{s,t}\qquad \text{ or }\qquad A_{s,t}^{-1} w_{10,1110} A_{s,t}$$
to represent the element $w_{s,t}$.
In what follows, any occurrence of $x_s,p_n,w_{s,t}$ shall be meant to be interpreted as this word.

The finite set of relations $\mathcal{R}_1$ consists of
\begin{enumerate}
\item $[xx_1^{-1}, x^{-1} x_1 x]=1\qquad [xx_1^{-1},x^2 x_1x^{-2}]$.
\item $x_1 c_3= c_2 x_2\qquad c_1x_0=c_2^2\qquad x_1c_2=c\qquad c^3=1$.
\item $w_{00,01}x_s=x_sw_{00,01}$ for $s\in \{1,11\}$.
\item $w_{00,011} x_{t}=x_{t} w_{00,011}$ for $t\in \{010, 0101\}\cup \{1,11\}$.
\item $w_{\sigma_1,\tau_1}w_{\sigma_2,\tau_2}=w_{\sigma_2,\tau_2}w_{\sigma_1,\tau_1}$\\ if $\sigma_1,\tau_1,\sigma_2,\tau_2$ are independent and $|\sigma_1|,|\tau_1|,|\sigma_2|,|\tau_2|\leq7$.
\item $w_{\sigma_1,\tau_1}w_{\sigma_2,\tau_2}=w_{\sigma_1,\tau_2}w_{\sigma_2,\tau_1}$\\ if $\sigma_1,\tau_1,\sigma_2,\tau_2$ are independent and $|\sigma_1|,|\tau_1|,|\sigma_2|,|\tau_2|\leq 7$.
\item $w_{\sigma,\tau}=x_{\sigma}w_{\sigma 0,\sigma 10}w_{\sigma 11,\tau}$ and $w_{\sigma,\tau}=x_{\tau}^{-1} w_{\sigma, \tau 1} w_{\tau 01,\tau 00}$\\ for $\sigma=10,\tau=110$ or $\sigma=10,\tau=1110$.
\item $w_{\sigma,\tau}w_{\tau,\nu}=w_{\sigma,\nu}$ if $\sigma,\nu$ are independent and $|\sigma|,|\tau|,|\nu|\leq 5$.
\item $w_{10,10}=1_S$.
\end{enumerate}

The generators $x, x_1, p_0$ together with the relations $(1)-(2)$ in $\mathcal{R}_1$ provide a finite presentation for the group $T$,
as is well known (See the discussion in pages 3-4 of \cite{BurilloClearyTaback}).
We will now show that the remaining relations in $\mathcal{R}$ can be reduced to the relations in $\mathcal{R}_1$.
We will need the following:

\begin{lem}\label{commutation}
Let $g\in T$ such that the relation $[w_{s,t},g]=1$ holds in $S$.
Then this relation can be expressed as a product of conjugates of relations $(1)-(4)$ in $\mathcal{R}_1$.
\end{lem}

\begin{proof}
We assume that the pair $00,01$ is in the $T$-orbit of $s,t$.
The case where $00,011$ is in the $T$-orbit of $s,t$ is similar.
Let $h=A_{s,t}^{-1}A_{00,01}$.
Then by the definitions above it follows that the relation $$[w_{s,t},g]=1$$
upon conjugation by $h$, reduces to a relation $$[w_{00,01}, g_1]=1$$
where $g_1= h^{-1} g h\in T$.

Now the subgroup of $T$ that commutes with the homeomorphism $w_{00,01}$
is the isomorphic copy of $F$ which is supported on $[10^{\infty},11^{\infty}]$.
Using the relations in $T$ (or $F$), we reduce $g_1$ to a word in the two generators $x_1,x_{11}$
of this subgroup.
Since these relations are in the family $(3)$ in $\mathcal{R}_1$, we are done.
\end{proof}

\begin{lem}
The family of relations $(4)-(5)$ in $\mathcal{R}$ can be expressed as a product of conjugates of relations $(1)-(4)$ in $\mathcal{R}_1$.
\end{lem}
\begin{proof}
Recall that these are the relations $w_{\sigma,\tau} f=f w_{\sigma \cdot f, \tau\cdot f}$ if $f\in T$ acts on $\sigma,\tau$.
Let $\sigma_1=\sigma \cdot f, \tau_1=\tau\cdot f$.
Assume that $10,110$ lies in the $T$-orbit of $\sigma,\tau$.
The case for $10,1110$ is similar.

Following the definitions above, we obtain that the relation equals 
$$A_{\sigma,\tau}^{-1} w_{10,110} A_{\sigma,\tau} f=f A_{\sigma_1,\tau_1}^{-1} w_{10,110} A_{\sigma_1,\tau_1}$$
Note that this is a commutation relation
$$[w_{10,110}, h]=1\qquad \text{ where } h= A_{\sigma_1,\tau_1} f^{-1} A_{\sigma,\tau}^{-1} $$
Since this is reducible to relations $(1)-(4)$ in $\mathcal{R}_1$ by Lemma \ref{commutation}, we are done.
 \end{proof}
 
 And now we end this subsection and the proof of Theorem \ref{main} with the following Lemma.
 
 \begin{lem}
 The relations $(6)-(10)$ in $\mathcal{R}$ can be expressed as conjugates of relations $(5)-(9)$ in $\mathcal{R}_1$.
 \end{lem}
 
 \begin{proof}
 We show this for the relations $(6)$ in $\mathcal{R}$. The proof for the remaining relations in $\mathcal{R}$ is similar.
 Recall the relations in $(6)$
 $$w_{\sigma_1,\tau_1}w_{\sigma_2,\tau_2}=w_{\sigma_2,\tau_2}w_{\sigma_1,\tau_1}$$ if $\sigma_1,\tau_1,\sigma_2,\tau_2$ are independent.
As an immediate consequence of Lemma \ref{Taction2}, there are $$s_1,s_2,t_1,t_2\in \{0^k1\mid 0\leq k\leq 6\}$$ and an $f\in T$
such that $$f(\sigma_1)=s_1, f(\sigma_2)=s_2,f(\tau_1)=t_1,f(\tau_2)=t_2$$
It follows that $$f^{-1} w_{\sigma_1,\tau_1}w_{\sigma_2,\tau_2} f=f^{-1} w_{\sigma_2,\tau_2}w_{\sigma_1,\tau_1}f$$
which reduces to
  $$w_{s_1,t_1}w_{s_2,t_2}=w_{s_2,t_2}w_{s_1,t_1}$$
  thanks to relations $(4)-(5)$ in $\mathcal{R}$ which have already been shown to be reducible to the relations in $\mathcal{R}_1$.
 \end{proof}

This concludes the proof of Theorem \ref{main}.

\end{proof}

\section{Simplicity of $S$}

In this section we shall prove that $S$ is simple.
We shall need the following observation.

\begin{lem}\label{Generation}
$S$ is generated by $G_0',T$.
\end{lem}

\begin{proof}
It suffices to show that $S$ is generated by $G_0',l$,
since $\langle F',l\rangle\cong T$.
In our original definition of $S$,
we fixed the generating set for $S$ as $$T\cup \{w_{10,110}\}$$
First we show that $\langle T,w_{10,110}\rangle<\langle G_0', l\rangle$.
Since $F'<G_0'$, and $F',l$ generate $T$, it follows that $T<\langle G_0', l\rangle$.
Furthermore, the generator $w_{10,110}$ equals $y_{10}y_{110}^{-1}$ as a homeomorphism, and so by Lemma \ref{charelementsG0prime} is an element of $G_0'$.
It follows that $\langle T, w_{10,110}\rangle < \langle G_0', l\rangle$.

Now we will show that $ \langle G_0', l\rangle< \langle T, w_{10,110}\rangle$.
We know that $l\in T$.
Recall that a standard form $fy_{s_1}^{t_1}...y_{s_n}^{t_n}$ is in $G_0'$
if and only if $s_1,...,s_n$ are non constant and $\sum_{1\leq i\leq n} t_i=0$.

First we make the following elementary observation.
If the standard form $fy_{s_1}^{t_1}...y_{s_n}^{t_n}$ in $G_0'$ satisfies that
$s_1,...,s_n$ are independent, then $fy_{s_1}^{t_1}...y_{s_n}^{t_n}$ is representable by a standard form in $\mathcal{S}$.
This is true since all percolating elements commute and
it is easy to express the homeomorphism $y_{s_1}^{t_1}...y_{s_n}^{t_n}$
as a word in the $w$-generators by pairing $y$-generators of positive and negative powers.

Now given any standard form $y_{s_1}^{t_1}...y_{s_n}^{t_n}\in G_0'$, we reduce it to a word of the above type by multiplying it with a 
suitable element of $S$ on the right.
This shall provide the desired outcome.

Using the fact that $s_1,...,s_n$ are non constant, we first find independent finite binary sequences  $k_1,...,k_{n}$ which satisfy that $s_1,...,s_n,k_1,...,k_n$ are independent.
Next, consider the word $w_{s_n,k_n}^{-t_n}...w_{s_1,k_1}^{-t_1}$.
It is easy to see that 
$$(y_{s_1}^{t_1}...y_{s_n}^{t_n})(w_{s_n,k_n}^{-t_n}...w_{s_1,k_1}^{-t_1})=y_{k_1}^{t_1}...y_{k_n}^{t_n}$$
and the right hand side is a word of the desired form.
\end{proof}

\begin{prop}
$S$ is simple.
\end{prop}

\begin{proof}
For any $g\in S\setminus \{1_S\}$, we will prove that the normal closure of $g$ is $S$.
Recall that $g$ is a homeomorphism of $\mathbf{S}^1=\mathbf{R}\cup \{\infty\}$.
First we show that the normal closure of $g$ contains an element $h$ whose support is contained in a compact interval in the real line.
Let $I\subset \mathbf{R}$ be a compact interval such that $g(I)\cup I$ is contained in a compact interval $J\subset \mathbf{R}$.
Let $f\in F$ be an element whose support lies in $I$.
The element $g^{-1} f^{-1} g f$ lies in the normal closure of $g$,
and moreover, it's support is contained in $g(I)\cup I$.
This is the required element $h$.

By Lemma \ref{stdform}, we know that $h$ can be represented as a standard form 
$$gw_{s_1,t_1}^{l_1}...w_{s_n,t_n}^{l_n}$$
Since $h$ fixes $\infty$, it follows that $g\in F$.
The literal translation of this word to $\mathcal{G}$ is the standard form
$g y_{s_1}^{l_1}y_{t_1}^{-l_1}...y_{s_n}^{l_n}y_{t_n}^{-l_n}$.
By Lemma \ref{charelementsG0prime}, we know that this standard form represents an element of $G_0'$.
It follows that $N\cap G_0'\neq \emptyset$.
Recall from Theorem \ref{simplegroup} that $G_0'$ is simple. 
This means that $F'<G_0'<N$ and $T\cap N\neq \emptyset$.
Since $T$ is simple, it follows that $T<N$.
By Lemma \ref{Generation}, it follows that $S=N$.
\end{proof}

\section{$C^1$ and piecewise linear actions on the circle}

The goal of this section is to prove the following:

\begin{prop}\label{lackofactions}
$S$ does not admit a non-trivial action by $C^1$-diffeomorphisms on the circle.  
$S$ does not admit a non-trivial action on the circle by piecewise linear homeomorphisms.
\end{prop}

First we need make a few preliminary observations about the subgroup structure of $S$.

\begin{lem}\label{G0subgroup}
$G_0<S$.
\end{lem}

\begin{proof}
For each element $f\in G_0$, we fix a standard form representation.
The choice of this standard form shall not be important,
however one can choose the \emph{normal form} described in \cite{Lodha} if one prefers.
We denote this set of standard forms as $\mathcal{J}$.

Recall that the group $F$ is self-similar.
In particular, for any finite binary sequence $s$,
the subgroup of $F$ which is supported on the interval $[s0^{\infty},s1^{\infty}]$
is isomorphic to $F$.
This isomorphism is canonical. 
The image of an element $f\in F$ under this isomorphism is described as the tree diagram of $f$ localised at the binary sequence $s$,
and the identity map outside $[s0^{\infty},s1^{\infty}]$.
We denote by $\mathbf{f}_s$ the image of $f$ under this isomorphism.
This informal notation shall only be used in this proof.

We define an injective homomorphism $\phi:G_0\to S$ as follows.
For $h\in G_0$, let $h=g y_{s_1}^{t_1}...y_{s_n}^{t_n}$ be the chosen standard form, and let $\sum_{1\leq i\leq n}t_n=t$.
Then $$\phi(h)=y_{110}^{-t} (\mathbf{g}_{10} y_{10 s_1}^{t_1}...y_{10 s_n}^{t_n})$$

\emph{Claim: This is an injective group homomorphism.}

First recall from \cite{LodhaMoore} that $G_0$ is generated by $x,x_1,y_{10}$, and from \cite{BurilloLodhaReeves} that its abelianisation is $\mathbf{Z}^3$.
Moreover, each element $y_s$ in $G_0$ is conjugate to $y_{10}$ in $G_0$.
It follows that for any standard form that equals the identity in $G_0$,
the sum of the powers of the $y_s$'s in the word must equal $0$.
Indeed the sum of the powers of the $y_s$'s is a homomorphism onto $\mathbf{Z}$.

Next, note that the restriction $$\phi(G_0)\restriction [100^{\infty},101^{\infty}]$$
is an isomorphism.
Additionally, the restriction $$\phi(G_0)\restriction [110 0^{\infty},110 1^{\infty}]$$ is a homomorphism $G_0\to \mathbf{Z}$.
An immediate consequence is that our map is a group homomorphism.
The fact that this is injective is clear.
\end{proof}

An immediate consequence is the following:

\begin{cor}\label{BBembedsinS}
$BB(1,2)<S$.
\end{cor}

\begin{proof}
The elements $x_{10},y_{100}, y_{101}$ of $G_0$ generate a copy of $BB(1,2)$ (see Section of \cite{BonattiLodhaTriestino} for the proof of this).
Since by Lemma \ref{G0subgroup} we know that $G_0<S$, we are done.
\end{proof}

\begin{proof}

{\bf Proof of Proposition \ref{lackofactions}}:

First note that since $S$ is simple, any action is either faithful or trivial.
We show that actions of either type mentioned in the statement cannot be faithful.

Recall from Theorem \ref{mainD} that the group $BB(1,2)$ does not admit a faithful $C^1$-action on the circle.
Moreover, from Lemma \ref{BBembedsinS} we know that $BB(1,2)<S$.
This proves that $S$ cannot admit a faithful $C^1$-action on the circle.

Consider an action of $BS(1,2)$ by piecewise linear homeomorphisms on the circle.
Then by Theorem \ref{power} it follows that there is an $m\in \mathbf{N}$ such that the subgroup $BS(1,2^m)$ admits a faithful piecewise linear action on a closed interval.
This is impossible because non abelian one relator groups do not embed in $PL_+([0,1])$.
(See Corollary $23$ in \cite{GubaSapir} for instance.)
\end{proof}

\begin{cor}\label{nonisom}
The group $S$ is not isomorphic to the following groups:
\begin{enumerate}
\item Thompson's group $T$.
\item The Higman-Thompson groups $T_n$.
\item The Brown-Stein groups $T(l,A,p)$.
\end{enumerate}
\end{cor}

In the second half of the proof of the Proposition above, the reader may wonder the following.
The most natural model of the groups $B(1,2)$ and $BB(1,2)$ arise as groups of piecewise linear homeomorphisms of the real line.
Then the question arises, why not extend the action to the compactification $\mathbf{S}^1=\mathbf{R}\cup \{\infty\}$, and regard it as an action 
by piecewise linear homeomorphisms of the circle?

We clarify that in the definition of a piecewise linear homeomorphism of the circle, (as defined by Thompson \cite{CannonFloydParry}, Brin \cite{BrinSquier}, Brown-Stein \cite{BrownStein} et. al.),
the map is built out of gluing restrictions of linear maps to closed intervals of the circle.
As a consequence, we are not allowed to glue restrictions of linear maps to open intervals, as in the case of $B(1,2),BB(1,2)$.
Indeed, it is not possible to view the group of orientation preserving piecewise linear homeomorphisms of the circle as a proper overgroup 
of the group of orientation preserving piecewise linear homeomorphisms of the real line.

\section{Nonamenable equivalence relations}

If a group $G$ acts on a set $X$, we denote by $E_G^X$ the associated orbit equivalence relation.
Let $S^1=\mathbf{R}\cup \{\infty\}$ be endowed with the usual Borel structure and let $E \subseteq S^1\times S^1$ be an equivalence relation
which is Borel and which has countable equivalence classes.
$E$ is \emph{$\mu$-amenable} if, after discarding a $\mu$-measure zero set,
$E$ is the orbit equivalence relation of an action of $\Zbb$.
(This is not the standard definition, but it is equivalent by \cite{Connes}.)
For the preliminaries surrounding this notion, we refer the reader to \cite{KechrisMiller}.
We fix $\mu$ as the Lebesgue measure on $\mathbf{S}^1$
The following was demonstrated in \cite{LodhaMoore} (See the discussion at the end of Section $2$).

\begin{lem}
$E_{G_0}^{\mathbf{S}^1}$ is a non $\mu$-amenable equivalence relation.
\end{lem}

As an immediate consequence, we obtain the following:

\begin{cor}
$E_S^{\mathbf{S}^1}$ is a non $\mu$-amenable equivalence relation.
\end{cor}

\begin{proof}
$E_{G_0}^{\mathbf{S}^1}\subset E_S^{\mathbf{S}^1}$ since $G_0<S$.
Non $\mu$-amenability of a subequivalence relation implies non $\mu$-amenability for the equivalence relation (See \cite{KechrisMiller} for instance).
\end{proof}

Compare this to the following:

\begin{lem}
Let $H$ be any finitely generated group of piecewise linear homeomorphisms of the circle.
Then $E_H^{\mathbf{S}^1}$ is an amenable equivalence relation.
\end{lem}

\begin{proof}
This is an immediate consequence of the fact that the affine group is metabelian and 
hence the actions of its finitely generated subgroups produce an amenable equivalence relation.
The equivalence relation given by a finitely generated group of piecewise linear homeomorphisms of the circle
is a subequivalence relation of one given by a finitely generated affine group.
\end{proof}

We now provide another proof of Corollary \ref{nonisom}.
Every group in the family $$\{T,T_n,T(l,A,p), S\}$$
acts on the circle in a \emph{locally dense} fashion.
As a consequence, Rubin's theorem holds. 
(We refer the reader to Section $9$ of \cite{Brin} for the statement of this theorem, as well as for the definition of locally dense).
If $S$ were isomorphic to another group in the family,
then by the theorem, the two actions would be topologically conjugate.
This is impossible, since the $\mu$-amenability of the equivalence relation is preserved under topological conjugation.
\section{Open questions}

This example was originally intended as a test case for the following open question (see Question $7.5$ in \cite{BekkaHarpeValette}).

\begin{question}
Is there an infinite group of homeomorphisms of the circle with Property $(T)$?
\end{question}

Navas showed that an infinite group acting by $C^{1+\alpha}$-diffeomorphisms of the circle cannot have property (T),
if $\alpha>\frac{1}{2}$.
(See Theorem $5.2.14$ in \cite{Navas}, for instance.)
Since $S$ cannot act by $C^1$ diffeomorphisms of the circle, the criterion of Navas \cite{Navas} cannot apply.

However, Yves de Cornulier and Nicol\'{a}s Matte Bon have pointed out that $S$ cannot have property $(T)$.
They provide an elementary argument, which we only briefly sketch below.
For a piecewise $C^1$- homeomorphism $f:\mathbf{S}^1\to \mathbf{S}^1$ and $x\in \mathbf{S}^1$,
we denote by $f_+^{'}(x), f_-^{'}(x)$ as respectively the right and left derivatives of $f$ at $x$.
The map $$b:S\to l^2(\mathbf{Q}\cup \{\infty\})\qquad f\to (\textup{log}\frac{f_+^{'}(x)}{f_-^{'}(x)})_{x\in \mathbf{Q}\cup \{\infty\}} $$
is then a non trivial cocycle that provides an affine isometric action of $S$ on $l^2(\mathbf{Q}\cup \{\infty\})$ without a global fixed point.
It is interesting to note that the image of the subgroup $T$ under this cocycle is trivial, since the action of $T$ is $C^1$.
This gives rise to the related question.

\begin{question}
Does the pair $(S,T)$ satisfy relative property $(T)$?
\end{question}

Another direction of inquiry is a question concerning rank of finitely presented simple groups.
A consequence of the classification of finite simple groups is that finite simple groups are always generatable by two elements.
It is natural to inquire the following:

\begin{question}
Is there a finitely presented infinite simple group which cannot be generated by two elements?
\end{question}

Recall that $T$ is generatable by two elements.
It follows that $S$ is generatable by three elements, since $S\cong \langle T, {\bf s}\rangle$.
The question ``\emph{is there a finitely generated infinite simple group which cannot be generated by two elements?}'' was stated in the Kourkova notebook \cite{Kourkova},
and it was solved by Guba in the affirmative \cite{Guba}.
In this context, we ask the following:

\begin{question}
Is $S$ generatable by $2$ elements?
What is the smallest number of relations needed to provide a presentation for $S$? 
\end{question}

\end{document}